\documentclass[12pt, final]{amsart}
\usepackage[top=4.2cm, bottom=4cm, left=2.4cm, right=2.4cm]{geometry}
\usepackage[utf8]{inputenc}
\usepackage[USenglish]{babel}
\usepackage[T1]{fontenc} 
\usepackage{mathrsfs}
\usepackage{mathtools}
\usepackage{amsmath}
\usepackage{amssymb,epsfig}
\usepackage{caption}
\usepackage{subcaption}
\usepackage{comment}
\usepackage{stmaryrd}
\usepackage{enumerate}
\usepackage[textsize=tiny]{todonotes}
\usepackage{amsthm}
\usepackage[absolute]{textpos}
\usepackage{mathtools,emptypage}

\usepackage[bookmarks=true]{hyperref}
\usepackage{xcolor}
\hypersetup{
    colorlinks,
    linkcolor={red!50!black},
    citecolor={blue!50!black},
    urlcolor={blue!80!black},
}

\usepackage[capitalise]{cleveref}
\usepackage{fixme}
\fxusetheme{color}

\newcommand{\AEm}{\textnormal{\AE}}

\newcommand{\N}{\mathbb{N}}
\newcommand{\Q}{\mathbb{Q}}
\newcommand{\R}{\mathbb{R}}

\newcommand{\sfd}{{\sf d}}

\renewcommand{\d}{{\mathrm d}}

\newcommand{\restr}[1]{\lower3pt\hbox{$|_{#1}$}}
\newcommand{\la}{{\langle}}                  
\newcommand{\ra}{{\rangle}}

\newcommand{\e}{{\rm{e}}}                          
\newcommand{\X}{{\rm X}}
\newcommand{\Y}{{\rm Y}}

\newcommand{\Lip}{\operatorname{Lip}}
\newcommand{\LIP}{\operatorname{LIP}}

\newcommand{\F}{{\rm F}}

\newcommand{\relMass}{\widetilde{\Mass{}}}

\newcommand{\kcur}{\mathbf{M}_k}
\newcommand{\ncur}{\mathbf{N}}
\newcommand{\kncur}{\mathbf{N}_k}
\newcommand{\mass}{\mathbb{M}}

\newcommand{\cM}{\mathbf{M}}
\newcommand{\cN}{\mathbf{N}}
\newcommand{\cF}{\mathbf{F}}
\newcommand{\cP}{\mathbf{P}}

\makeatletter
\newcommand{\mytag}[2]{%
  \text{#1}%
  \@bsphack
  \begingroup
    \@onelevel@sanitize\@currentlabelname
    \edef\@currentlabelname{%
      \expandafter\strip@period\@currentlabelname\relax.\relax\@@@%
    }%
    \protected@write\@auxout{}{%
      \string\newlabel{#2}{%
        {\color{black}#1}%
        {\thepage}%
        {\@currentlabelname}%
        {\@currentHref}{}%
      }%
    }%
  \endgroup
  \@esphack
}
\makeatother

\def\XXint#1#2#3{{\setbox0=\hbox{$#1{#2#3}{\int}$ }
\vcenter{\hbox{$#2#3$ }}\kern-.6\wd0}}

\newcommand{\curr}[1]{\llbracket #1 \rrbracket}
\newcommand{\Mass}[1]{\mathbb{M} #1}
\newcommand{\B}{\mathbb{B}}

\newcommand{\Poly}[1]{\mathcal{P}_{#1}}

\newcommand{\Flatnorm}{\mathbb{F}}

\allowdisplaybreaks

\newtheorem{theorem}{Theorem}[section]

\newtheorem{corollary}[theorem]{Corollary}
\newtheorem{lemma}[theorem]{Lemma}
\newtheorem{proposition}[theorem]{Proposition}
\theoremstyle{definition}
\newtheorem{definition}[theorem]{Definition}

\newtheorem{conjecture}[theorem]{Conjecture}

\newcounter{Counter}

\newtheorem{remark}[theorem]{Remark}

\title[Structure of Metric $1$-currents]{Structure of Metric $1$-currents: approximation by normal currents and representation results}

\author{David Bate}
\author{Emanuele Caputo}
\author{Jakub Tak\'{a}\v{c}}
\author{Phoebe Valentine}
\author{Pietro Wald}


\begin{document}

\begin{abstract}
    We prove the $1$-dimensional flat chain conjecture in any complete and quasiconvex metric space, namely that metric $1$-currents can be approximated in mass
    by normal $1$-currents.
    The proof relies on a new Banach space isomorphism theorem, relating metric $1$-currents and their boundaries to the Arens-Eells space.

    As a by-product, any metric $1$-current in a complete and separable metric space
    can be represented as the integral superposition of oriented $1$-rectifiable sets, thus dropping a finite dimensionality condition from previous results of Schioppa \cite{Schioppa_derivations, Schioppa_currents}.

    The connection between the flat chain conjecture and the representation result
    is provided by a structure theorem for metric $1$-currents in Banach spaces,
    showing that any such current can be realised as the restriction to a Borel set of a boundaryless normal $1$-current. This generalizes, to any Banach space, the $1$-dimensional case of a recent result of Alberti-Marchese in Euclidean spaces \cite{AlbertiMarcheseFlatChains}.
    The argument of Alberti-Marchese requires the strict polyhedral approximation theorem of Federer for normal $1$-currents, which we obtain in Banach spaces.
\end{abstract}

\maketitle

\tableofcontents

\section{Introduction}

The need for a generalized notion of a surface emerged from the desire to tackle geometric variational problems which resisted smooth techniques, most notably Plateau's problem. Federer and Fleming realised that de Rham's notion of currents, which was developed in duality with differential forms, was in fact a compactification of the class of oriented smooth manifolds embedded in $\R^n$ \cite{FedererFleming1960} and so provided a notion of a generalised surface highly suitable to such variational problems. We shall refer to these currents as \textit{classical currents}.

To tackle such problems in metric spaces, Ambrosio and Kirchheim defined a notion of \textit{metric currents} for ambient metric spaces in \cite{AK00}, but their construction is not a direct generalisation of a classical current because it is necessary to make stronger requirements of a metric current to account for the reduced structure of the space. 
Indeed, the theory of classical currents is developed in duality with differential forms, however, for a metric space definition to be suitably general, it is necessary to move away from the strong assumption of a differentiable structure and thus differential forms. 
Instead, we consider $(k+1)$-tuples of Lipschitz functions $(f,\pi_1,\dots,\pi_k)$, where $f$ is bounded; a collection which we denote $D^k(\X)$. 
A $k$-metric current, for $k \in\N$, is then defined as a functional on $D^k(\X)$ which is multilinear, jointly continuous, local and satisfies a finite mass condition (all to be explicitly defined in Section \ref{sec:preliminaries}). The locality in particular ensures that the functional depends in a weak sense on the derivatives of the last $k$ entries, echoing the action on the differential form $f\d \pi_1 \wedge \dots \wedge\d \pi_k$. We will denote the space of metric $k$-currents on $\X$ by $\cM_k(\X)$.

Metric currents have been used successfully in various problems in analysis and geometry (\cite{BassoMartiWenger2023,LangKleiner2020,Song2024,Wenger2011}). However, their structure, even on Euclidean space $\R^n$, is only partially understood. The most famous conjecture in this vein was formulated in Ambrosio and Kirchheim's foundational paper and is now known as the flat chain conjecture \cite[Pag. 68]{AK00}.
The question is whether metric $k$-currents on $\R^n$ correspond to classical $k$-flat chains for $1 \le k\le n$. 
Partial progress has been made: positive results are known for $k=1$ and $k=n$. The latter is a corollary of a deep result by De Philippis and Rindler \cite{DePhilippisRindler_A_free} where they prove that $n$-dimensional metric currents in $\R^n$ correspond to the space of Lebesgue integrable functions. 
The case of $k=1$ was proved by Schioppa in \cite{Schioppa_currents}. 
The pivotal tool in his proof is a powerful representation result, which, in the case of $k=1$, states that every metric $1$-current can be written as a superposition of oriented $1$-rectifiable sets. 
This representation result holds in metric spaces satisfying a technical finite-dimensionality condition which is satisfied, for example, if the ambient space is metrically doubling. Due to the generality of his representation result, Schioppa in fact proves something more general than the Euclidean flat chain conjecture. The author proves that normal $1$-currents are flat (in fact, mass) dense in the space of metric $1$-currents in quasiconvex metric spaces satisfying the same standing technical assumption.
\par 
In Euclidean space, it is known that normal currents are flat chains, thus Schioppa's density result is equivalent to the flat chain conjecture. Schioppa's result can thus be seen as answering a metric version of the conjecture. Thus the metric analogue of the flat chain conjecture asks whether normal currents are dense under the flat norm in the space of metric currents. 

\begin{conjecture}[Flat chain conjecture]
Given $k \in \N$, every $k$-current on a metric space $\X$ can be approximated by normal $k$-currents under the flat norm:
    \[\overline{\{\text{normal }k\text{-currents}\}}^{\text{Flat norm}}=\kcur(\X).\]
\end{conjecture}

It turns out that a connectivity assumption is inherently required: for example, a Cantor set of positive Lebesgue measure in the real line is a metric space which supports no non-trivial normal $1$-currents, but the space of metric $1$-currents is rich. However the finite dimensionality assumption is not apriori necessary. The two main result of this work are the $1$-dimensional case of the flat chain conjecture and the representation formula for metric $1$-currents as a superposition of oriented $1$-rectifiable sets. Both results are obtained in complete metric spaces without the finite dimensionality assumption. Quasiconvexity of the underlying space is required for the flat chain conjecture, while the representation formula holds in any complete metric space.

\textbf{A note on related work.}
While writing this paper, we became aware that Arroyo-Rabasa and Bouchitté have simultaneously and independently obtained results related to those presented here \cite{ArroyoRabasaBouc2025}. 
To the best of our knowledge, their techniques draw on optimal transport and functions of bounded variation, whereas ours are typically linked to metric geometry and analysis on metric spaces.

\subsection{Flat chain conjecture}

Our approach to the $1$-dimensional case of the flat chain conjecture is based on the study of the Arens-Eells space, a Banach space that we denote by $(\AEm(\X),\|\cdot\|_{\AEm(\X)})$.
It can be defined as the completion of \emph{molecules}, namely the vector space of finitely supported
real-valued
measures $m$ with $m(\X)=0$, equipped with the norm
\begin{equation*}
\|m\|_{\AEm(\X)}:=\inf\left\{\sum_{i=1}^n|a_i|\sfd(x_i,y_i)\colon m=\sum_{i=1}^na_i(\delta_{x_i}-\delta_{y_i})\right\},
\end{equation*}
where the infimum is taken over all $n\in\N$, $a_i\in\R$, and $x_i,y_i\in\X$
such that
\begin{equation*}
    m=\sum_{i=1}^na_i(\delta_{x_i}-\delta_{y_i}).
\end{equation*}
Importantly, the Arens-Eells space may be characterised as the predual
of the Banach space of Lipschitz functions vanishing at a fixed distinguished point,
normed with the Lipschitz constant; see \cite{WeaverBook}.
As such, elements of $\AEm(\X)$ can be seen as acting on Lipschitz functions
similarly to boundaries of metric $1$-currents.
In fact, for every $T\in\cM_1(\X)$, we may regard $\partial T$ as an element of
$\AEm(\X)$ and moreover $T\in\cM_1(\X)\mapsto\partial T\in\AEm(\X)$ is a bounded linear operator\footnote{For this reasoning to be valid, it is necessary to also assume that $\X$ is separable. The separability assumption can, however, be omitted, under the assumption that every cardinality is an Ulam number, an axiom consistent with ZFC. A large portion of the theory of Ambrosio and Kirchheim makes this assumption; see page 13 of their manuscript \cite{AK00}. Lastly, we remark that if one defined joint continuity in \emph{topological} terms, instead of relying on \emph{sequences}, the separability assumption could also be lifted and in the separable case one would obtain the same theory.}.
As will be clear shortly, for the purpose of studying the flat chain conjecture,
it is of interest to determine when $\partial$ is a (Banach space) isomorphism.
Related to this is the recent work of De Pauw \cite{DePauw} where a connection of the Arens-Eells space to \emph{flat $0$-chains} is explored. For us, on the other hand, the point is to study specifically the relationship of the Arens-Eells space to the space of \emph{metric $1$-currents of finite mass}, which is what we achieve in our first main result.

\begin{theorem}[Isomorphism theorem, cf.\ Theorem \ref{thm:isomorphism}]
\label{thm:isomorphism_intro}

Let $(\X,\sfd)$ be a complete and separable metric space.
Then 
\[\partial\colon \cM_1(\X)/\mathrm{ker}(\partial)\to \AEm(\X)\]
is a Banach space isomorphism if and only if $\X$ is quasiconvex.
In the latter case, we have
\begin{equation}
\label{eq:1.2inThm}
    \mathrm{qc}(X)^{-1}\|\partial T\|_{\AEm(\X)}\leq \|[T]\|\leq \|\partial T\|_{\AEm(\X)} \quad\text{for every $T\in\cM_1(\X)$,}
\end{equation}
with the constants in \eqref{eq:1.2inThm} being optimal.\footnote{Unless the metric space $\X$ consists of a point,
in which case the optimal constants are $0$.} Here $\mathrm{qc}(\X):=\inf\{C\in[1,\infty]\colon\X\text{ is }C\text{-quasiconvex}\}$
and $[T]\in\cM_1(\X)/\ker(\partial)$ denotes the equivalence class of $T\in\cM_1(\X)$
in $\cM_1(\X)/\ker(\partial)$.
\end{theorem}

Given a metric current $T\in\cM_1(\X)$, $\partial T$
may be approximated in $\AEm(\X)$ by a sequence of molecules $\{m_i\}_i$.
The invertibility of $\partial\colon\cM_1(\X)/\ker(\partial)\to\AEm(\X)$
ensures that we can find $T_i\in\cM_1(\X)$ with $\partial T_i=m_i$ so that
$[T_i-T]\rightarrow 0$ in the norm of $\cM_1(\X)/\ker(\partial)$.
In particular, $\{T_i\}_i$ is a sequence of normal $1$-currents and, possibly adding
boundaryless currents to $T_i$, we are then able to approximate $T$ in mass
with normal $1$-currents.
Thus, the following approximation result is a corollary of Theorem \ref{thm:isomorphism_intro}. 

\begin{theorem}[cf.\ Corollary \ref{coro:approximation_by_normal_currents}]
\label{flat_chain_intro}
Let $(\X,\sfd)$ be a complete and quasiconvex metric space.
Let $T\in\cM_1(\X)$ and suppose its mass measure is inner regular by compact sets.
Then there is a sequence $T_i$ of normal $1$-currents such that
$\Mass{(T-T_i)}\rightarrow 0$.

\end{theorem}

\subsection{Old and new representation results}

The simplest and most concrete examples of metric $1$-currents consist of Lipschitz curves.
For any metric space $\X$ and Lipschitz curve $\gamma\colon[0,1]\to\X$,
\begin{equation*}
    \curr{\gamma}(f,\pi):=\int_0^1(f\circ\gamma)_t(\pi\circ\gamma)_t'\,\d t, \qquad (f,\pi)\in D^1(\X),
\end{equation*}
defines a metric $1$-current.
It is then natural to wonder how far can a metric $1$-current be from a linear combination
of currents as above.
More precisely, one might look for `concrete' representations of metric $1$-currents in terms
of Lipschitz curves.

Such questions are not new and a lot is already known.
In Euclidean space, Smirnov \cite{Smirnov1993}
proves several representation results for normal $1$-currents, which
were later extended to metric spaces by Paolini-Stepanov \cite{PaoliniStepanov_acyclic, PaoliniStepanov_cycles}.
In particular, they prove that, on a complete and separable metric space $\X$, any normal $1$-current $N\in\cN_1(\X)$
can be represented as
\begin{equation}
\label{eq:normal_decomposition}
    N =\int\curr{\gamma}\,\d\eta(\gamma), \qquad
\Mass(N)=\int\Mass(\curr{\gamma})\,\d\eta(\gamma),
\end{equation}
where $\eta$ is a finite Borel measure on the space of Lipschitz curves; see \cite{PaoliniStepanov_acyclic,PaoliniStepanov_cycles} for details.
(Representations with good properties of the boundary are also available;
see \cite{PaoliniStepanov_cycles, Smirnov1993}.
See also \cite{Ambrosio_Renzi_Vitillar_local_Smirnov} for recent
results on locally normal $1$-currents.)\par
For integral $1$-currents there are even better representations.
In \cite[Thm.\ 5.3]{BonicattoDelNinPasqualetto2022}, it is proven that in any complete metric space, for every integral $1$-current $T$, there exists an at most countable collection $\{\gamma_i\}_{i\in I}$ of Lipschitz curves, injective outside of the endpoints, such that
\begin{equation}\label{eq:integral_decomposition}
    T=\sum_{i\in I} \curr{\gamma_i},\quad \Mass(T)=\sum_{i \in I} \Mass(\curr{\gamma_i}),\quad \Mass(\partial T)=\sum_{i \in I} \Mass(\partial \curr{\gamma_i}).
\end{equation}

For general $1$-currents, representations such as \cref{eq:normal_decomposition}
are not possible, at least
if one insists to consider curves $\gamma\colon[0,1]\to\X$.
To see this, let $K\subset\R$ be a totally disconnected compact set of positive Lebesgue measure,
and consider the $1$-metric current $\curr{K}$ on $\R$ defined as
$\curr{K}(f,\pi):=\int_{K}f(t)\pi'(t)\,\d t$ for $(f,\pi)\in D^1(\R)$.
It is then clear that, for $T=\curr{K}$, \cref{eq:normal_decomposition} holds for no $\eta$
on rectifiable curves.
To recover a representation result for general $1$-currents,
it is then necessary to consider a larger family of building blocks.
A viable option consists of \emph{curve fragments}, namely
Lipschitz functions $\gamma\colon {\rm dom}(\gamma)\to\X$, where ${\rm dom}(\gamma)$ is a compact
subset of $\R$; see Section \ref{sec:preliminaries} for the precise definition.
Indeed, combining results of Schioppa, it is possible to show that, in some cases,
an analogue of \cref{eq:normal_decomposition} holds for general
metric $1$-currents, provided we replace rectifiable curves with curve fragments.
More precisely, if $\X$ is a complete and separable metric space and $T\in\cM_1(\X)$,
\cite[Theorem 3.7]{Schioppa_currents} together with \cite[Theorem 3.98]{Schioppa_derivations}
provide sufficient conditions\footnote{Schioppa proves that the action of a current is given by differentiation along curve fragments, as in \eqref{eq:decomposition_Schioppa}, whenever the module of derivations associated to 
the mass measure $\|T\|$ is finitely generated; see \cite{Schioppa_currents, Schioppa_derivations} for more details.
This condition is satisfied if $\X$ is doubling (see \cite[Corollary 5.99]{Schioppa_derivations} or \cite[Corollary 4.13]{Schioppa_currents})
or more generally if $\X$ has finite Hausdorff dimension
(combine \cite[Theorem 5.3]{BKO23} with \cite[Corollary 3.93]{Schioppa_derivations}).}
for the existence
of a finite Borel measure $\eta$ concentrated on biLipschitz curve fragments
and a Borel measurable function $(t,\gamma)\mapsto G(t,\gamma)\in\R$ so that
\begin{equation}\label{eq:decomposition_Schioppa}
\begin{aligned}
T(f,\pi)&=\int\int_{{\rm dom}(\gamma)} (f\circ\gamma)_t(\pi\circ\gamma)_t'\,G(t,\gamma)\,\d t\,\d\eta(\gamma), &(f,\pi)\in D^1(\X), \\
\|T\|(B)&=\int\int_{{\rm dom}(\gamma)}(\chi_B\circ\gamma)_t\, G(t,\gamma)\,\d t\, \d\eta(\gamma),
&B\subset\X \text{ Borel},
\end{aligned}
\end{equation}
where the fact that the expressions are well-defined is part of the thesis.
See Section \ref{sec:preliminaries_currents} for the definition
mass measure $\|T\|$.
\par
The theorems mentioned so far sit in a series of results where better representations are obtained by increasing the regularity of the current and, for normal and integral currents,
only completeness and separability of the underlying metric space are required.
This was a further motivation to prove a representation result for general metric $1$-currents in the same generality, thus removing the additional assumptions from \cite{Schioppa_derivations,Schioppa_currents}.

\begin{theorem}[Representation of metric $1$-currents, cf. Theorem \ref{thm:representation_currents}]
\label{thm:representation_currents_preliminaries}
Let $(\X,\sfd)$ be a complete and separable metric space. Let $T \in \cM_1(\X)$. Then there exists a finite Borel measure $\eta$ on curve fragments such that
\begin{equation}\label{eq:main_result_representation_intro}
    T=\int \curr{\gamma} \,\d \eta(\gamma)\quad\text{ and }\quad\Mass(T)=\int \Mass(\curr{\gamma})\,\d \eta(\gamma).
\end{equation}
\end{theorem}
We point out that equations \eqref{eq:main_result_representation_intro},
\eqref{eq:normal_decomposition}, \eqref{eq:integral_decomposition}
imply an equality analogous to the second line of \eqref{eq:decomposition_Schioppa};
see Remark \ref{rmk:mass_vs_mass_on_curves}.



All the three representations share two similar properties:
\begin{itemize}
    \item The action of the metric current can be recovered by a suitable superposition, countable for the integral case and possibly uncountable for normal $1$-currents and metric $1$-currents, of $1$-dimensional objects. These objects are $1$-rectifiable sets for the case of metric $1$-currents and rectifiable curves for the normal and integral case;
    \item The mass of the current is equal to the sum (or the integral) of the the mass of the $1$-dimensional objects, be it curves or $1$-rectifiable sets.
\end{itemize}

\subsection{Main ideas on the proof of the representation theorem}

We comment on the tools needed to prove Theorem \ref{thm:representation_currents_preliminaries}. The proof is divided in two steps.

\textbf{Step 1: Generalisation of a result of Alberti-Marchese to Banach spaces.}
The starting point is \cite[Theorem 1.1]{AlbertiMarcheseFlatChains},
where it is shown that for any $\varepsilon>0$ and classical $k$-dimensional flat chain $T$ in $\R^n$, $1\leq k<n$,
there is a boundaryless normal $k$-current $N$ and a Borel set $B$
so that $T=N\restr{B}$ and $\mass(N)\leq(2+\varepsilon)\mass(T)$.

Using Theorem \ref{flat_chain_intro}, we are able to extend the
$1$-dimensional case of \cite[Theorem 1.1]{AlbertiMarcheseFlatChains}
to arbitrary Banach spaces,
showing that any metric $1$-current is the restriction of a normal $1$-current as above. 
For our application (see Step 2), we need to be able to take $B$ closed. We prove that this is possible if we additionally assume that the metric $1$-current $T$ sits in an affine hyperplane in the ambient Banach space.

The proof of \cite{AlbertiMarcheseFlatChains} crucially relies on
the classical polyhedral approximation theorem \cite[Theorem 4.1.23]{Federer69},
which is ultimately based on the fact that, for $N\in\cN_1(\R^n)$,
there are polyhedral chains $\{P_i\}_i$
so that 
\begin{equation*}
    \Flatnorm(N-P_i) \to 0\quad \text{and} \quad \Mass(P_i) \to \Mass(N).
\end{equation*}
For $k=1$, we provide a proof of the above in metric spaces
admitting a conical geodesic bicombing (which include Banach and ${\sf CAT}(0)$ spaces, see
\cite{DescombesLang2015} and Section \ref{sec:deformations}),
where polyhedral chains are replaced with chains of geodesics
(segments if the space is linear); see Theorem \ref{thm:geodesic_approximation}.
For $k\geq 2$, the above strict approximation result fails in every infinite-dimensional
Hilbert space and, therefore, it is currently unclear if
\cite[Theorem 1.1]{AlbertiMarcheseFlatChains}
can be extended beyond finite dimensional Euclidean space. The failure of the approximation result for $k\geq 2$ is the subject of an ongoing work of the authors. Since this result is inconsequential for the present paper, we omit further comments.

\textbf{Step 2: Application of Paolini-Stepanov representation theorem.}
Given a metric $1$-current $T$ on a complete and separable metric space $\X$, we embed $\X$ isometrically in a separable Banach space $\B$ and we consider the push-forward current. Up to enlarging the Banach space, we may assume that $T$ is supported in an affine hyperplane of $\B$. We now apply Step 1 in $\B$ and find a boundaryless normal $1$-current $N\in\cN_1(\B)$ and a closed set $C$
so that $N\restr{C}=T$.
Finally, restricting the representation \eqref{eq:normal_decomposition}
of $N$ to $C$, proves \eqref{eq:main_result_representation_intro} for $T$.
The fact that $C$ is closed ensures that the restricted curve fragments are still
curve fragments, i.e.\ have compact domain. Thus, the representation in Theorem \ref{thm:representation_currents_preliminaries} is given by intersecting the rectifiable curves of the representation of $N$ with the set $C$.

\subsection{Structure of the paper}

The paper is structured as follows. Section \ref{sec:preliminaries} contains preliminaries about curves, fragments, the Arens-Eells space and metric currents.
Section \ref{sec:flat_chain_conjecture} contains the proof of the isomorphism theorem and the flat chain conjecture in complete quasiconvex metric spaces. In Section \ref{sec:deformations}, we prove the homotopy formulas for $1$-dimensional normal current in conical geodesic bicombing metric spaces. In Section \ref{sec:alberti_marchese}, we generalise the result in \cite{AlbertiMarcheseFlatChains} to Banach spaces and in Section \ref{sec:representation_results} we prove the representation result in complete and separable metric spaces. Appendix \ref{appendix:A} and \ref{appendix:B} contain some technical results about metric currents.

\subsection*{Acknowledgments}

D.B., E.C. and J.T. are supported by the European Union’s Horizon 2020 research
and innovation programme (Grant agreement No. 948021).
J.T., P.V. and P.W. are supported by the Warwick
Mathematics Institute Centre for Doctoral Training and the UK Engineering and Physical Sciences Research Council (Grant number: EP/W524645/1).

\section{Preliminaries}
\label{sec:preliminaries}

Let $(\X,\sfd)$ be a metric space. Given $A \subset \X$, we define 
${\rm diam}(A,\sfd):=\sup \{ \sfd(x,y):\, x,y \in A \}$.
We denote with $\mathcal{M}(\X)$ the convex cone of nonnegative finite Borel measures on $\X$.

Let $C([0,1],\X)$ be the space of continuous curves $\gamma \colon [0,1] \to \X$,
equipped with the uniform distance $\sfd_\infty(\gamma,\gamma'):=\max_{t\in [0,1]}d(\gamma_t,\gamma'_t)$,
where $\gamma_t$ denotes the value of $\gamma$ at $t$.
We also set, for $\gamma\in C([0,1],\X)$, 
\begin{equation*}
    \ell(\gamma):=\sup \left\{ \sum_{i=1}^{N} \sfd(\gamma_{t_{i-1}},\gamma_{t_i})\,:\,a=t_0 < t_1 < \dots < t_N=b,\, N \in \mathbb{N}
    \right\}
\end{equation*}
and call it \emph{length} of $\gamma$.
We say that $\gamma\in C([0,1],\X)$ is rectifiable whenever $\ell(\gamma)< \infty$.

Given a metric space $(\X,\sfd)$, we define the \emph{intrinsic distance} $\sfd_\ell \colon \X \times \X \rightarrow [0,\infty]$ as 
\begin{equation*}
    \sfd_\ell(x,y):=\inf\left\{ \ell(\gamma): \, \gamma \in C([0,1],\X),\, \gamma_0=x,\, \gamma_1=y \right\}.
\end{equation*}
Notice that we do not assume $\X$ to be rectifiably path-connected, so $\sfd_\ell$ may assume extended values.

We say that $\gamma \in C([0,1],\X)$ is an \emph{$L$-quasigeodesic} for $L \ge 1$ if $\ell(\gamma)\le L \sfd(\gamma_0,\gamma_1)$, and \emph{geodesic} if it is $1$-quasigeodesic.
We denote the space of all geodesics by ${\rm Geo}(\X)$.

Given $L\ge 1$, we say that a metric space $(\X,\sfd)$ is \emph{$L$-quasiconvex} if for every $x,y\in \X$ there exists an $L$-quasigeodesic $\gamma$ such that $\gamma_0=x$ and $\gamma_1=y$. We say that a metric space is \emph{geodesic} if it is $1$-quasiconvex.

The space of \emph{curve fragments} on a metric space $\X$ is defined as
\begin{equation*}
    \Gamma(\X):=\{\gamma\colon K\to \X\colon K\subset [0,1]\text{ compact}, \gamma\ 1\text{-Lipschitz}\}
\end{equation*}
and, for $\gamma\in\Gamma(\X)$, we let ${\rm dom}(\gamma)$ denote its domain,
so that $\gamma\colon{\rm dom}(\gamma)\to\X$.
Further, given a set $E\subset\X$, we denote 
\[\Gamma(\X,E):=\{\gamma\in\Gamma(\X)\colon {\rm im}(\gamma)\cap E\neq \varnothing\},\]
where
${\rm im}(\gamma)$ denotes the image of $\gamma$.
Letting ${\rm gr}(\gamma):=\{(t,\gamma_t)\in \mathbb{R}\times \X:t\in {\rm dom}(\gamma)\}$ denote the graph of $\gamma$, we endow $\Gamma(\X)$ with the topology induced by the Hausdorff distance on
the graphs in $\R\times\X$.\par

For every $\gamma \in \Gamma(\X)$, the limit
\begin{equation*}
    |\dot{\gamma}_t|:=\lim_{{\rm dom}(\gamma)\ni s \to t} \frac{\sfd(\gamma_s,\gamma_t)}{|t-s|}
\end{equation*}
exists for a.e.\ $t \in {\rm dom}(\gamma)$ and is called the \emph{metric derivative} of $\gamma$ at time $t$.

We define the length of a curve fragment $\gamma\in\Gamma(\X)$ as
\begin{equation*}
    \ell(\gamma):=\int_{{\rm dom}(\gamma)} |\dot{\gamma}_t|\,\d t,
\end{equation*}
which is consistent with the previous definition of length of curve.

\begin{lemma}\label{lemma:meas_restriction_frag}
Let $(\X,\sfd)$ be a complete and separable metric space, $E\subset \X$ a non-empty $\sigma$-compact set, and $F:=\overline{E}$.
Then $\Gamma(\X,E)$ is a Borel subset of $\Gamma(\X)$ and
the map
$\Phi\colon\Gamma(\X,E)\to \Gamma(F)$ given by $\Phi(\gamma):=\gamma\vert_{\overline{\gamma^{-1}(E)}}$ is a Borel map satisfying
$\gamma\vert_{\gamma^{-1}(E)}\subset\Phi(\gamma)\subset\gamma\vert_{\gamma^{-1}(F)}$.
\end{lemma}
\begin{proof}
We let $\mathcal{C}(\X)$ denote the collection of closed and bounded subsets of $\X$, endowed with the Hausdorff distance.
Let $\{K_n\}_n$ be a non-decreasing sequence of non-empty compact sets such that $E=\cup_nK_n$.
By \cite[Lemma 3.3]{bate2025albertirepresentationsrectifiabilitymetric} and continuity of the map $\gamma\in \Gamma(\X)\mapsto{\rm im}(\gamma)\in\mathcal{C}(\X)$, $\Gamma(\X,K_n)$
is closed, so $\Gamma(\X,E)=\cup_n\Gamma(\X,K_n)$ is Borel.

Fix $\gamma_0\in\Gamma(F)$ and define $\Phi_n\colon\Gamma(\X,E)\to\Gamma(F)$
as $\Phi_n(\gamma):=\gamma_0$ for $\gamma\notin\Gamma(\X,K_n)$
and $\Phi_n(\gamma):=\gamma\vert_{\gamma^{-1}(K_n)}$ otherwise.
By \cite[Lemma 2.22]{Schioppa_derivations}, $\Phi_n\vert_{\Gamma(\X,K_n)}$ is Borel,
which implies that $\Phi_n$ is Borel measurable.
We claim that $\Phi_n$ converges pointwise to the map
$\Phi(\gamma):=\gamma\vert_{\overline{\gamma^{-1}(E)}}$, $\gamma\in\Gamma(\X,E)$.
Fix $\gamma\in\Gamma(\X,E)$ and let $n_0\in\N$ so that
${\rm im}(\gamma)\cap K_n\neq\varnothing$ for $n\geq n_0$.
Let $\varepsilon>0$, $\mathcal{T}\subset \overline{\gamma^{-1}(E)}\subset[0,1]$
a finite $\varepsilon$-net in $\overline{\gamma^{-1}(E)}$, and note that we may assume
$\mathcal{T}\subset\gamma^{-1}(E)$.
Since $\{K_n\}_n$ is non-decreasing and $E=\cup_nK_n$, there is $n_1\geq n_0$
so that $\mathcal{T}\subset\gamma^{-1}(K_n)$ for $n\geq n_1$.
Then, by $\LIP(\gamma)\leq 1$, we have ${\rm gr}(\Phi(\gamma))\subset B\big({\rm gr}(\Phi_n(\gamma)),\varepsilon\big)$ for $n\geq n_1$,
where $[0,1]\times\X$ is endowed with the distance
$((t_1,x_1),(t_2,x_2))\mapsto\max\{|t_1-t_2|,\sfd(x_1,x_2)\}$.
The above inclusion together with ${\rm gr}(\Phi_n(\gamma))\subset {\rm gr}(\Phi(\gamma))$
(for $n\geq n_0$)
conclude the proof of the claim and hence the lemma.
\end{proof}

We denote by $\mathcal{B}^\infty(\X)$ the space of bounded Borel functions from $\X$ to $\R$. Given two metric spaces $(\X,\sfd_\X)$ and $(\Y,\sfd_\Y)$, for any $f\colon \X \to \Y$, we define the \emph{Lipschitz constant}
\begin{equation*}
    \LIP(f):=\sup_{y\neq x}\frac{\sfd_\Y(f(x),f(y))}{\sfd_\X(x,y)}.
\end{equation*}
We denote the set of functions from $\X$ to $\Y$ whose Lipschitz constant is finite with $\Lip(\X;\Y)$ which we equip with the seminorm $\LIP(\cdot)$.
We will also use the simpler notation $\Lip(\X)$ when $Y=\mathbb{R}$. 

Given a distinguished point $0\in \X$, we denote by $\Lip_0(\X)$,  $\Lip_b(\X)\subset\Lip(\X)$ the space of functions vanishing at the point $0$ and the space of bounded Lipschitz functions respectively.

\begin{remark}
\label{rem:topologies_agree}
We can endow $\Lip_1([0,1],\X)$, the set of $1$-Lipschitz
functions $[0,1]\to\X$, with the subspace topology of either $C([0,1],\X)$ or $\Gamma(\X)$.
However, it is not difficult to verify that they agree.
\end{remark}

For $f\colon (\X,\sfd_\X)\to (\Y,\sfd_\Y)$
we define its \emph{pointwise Lipschitz constant} at $x\in\X$ as
\begin{equation*}
    \Lip f(x):=\limsup_{r\rightarrow 0}\sup_{y\in B(x,r)}\frac{\sfd_\Y(f(x),f(y))}{r}=\limsup_{y\rightarrow x}\frac{\sfd_\Y(f(x),f(y))}{\sfd_\X(x,y)},
\end{equation*}
if $x$ is a limit point of $\X$
and $\Lip f(x):=0$ otherwise.
If $f$ is Lipschitz, then $\Lip f$ is Borel measurable.
Similarly, we define
the \emph{asymptotic Lipschitz constant} of $f$ as
as $\Lip_a f(x):=\lim_{r\rightarrow 0}\LIP(f\restr{B_r(x)})=\inf_{r>0} \LIP (f \restr{B_r(x)})$ if $x$ is a limit point, while $\Lip_a f(x):=0$ otherwise.

\subsection{Arens-Eells space}
Let $(\X,\sfd)$ be a metric space.
We call a measure $m$ on $\X$ a \emph{molecule} if it is a finite linear combination of Dirac measures with zero average. Any molecule $m$ is of the form
\begin{equation*}
    m=\sum_{i=1}^N \lambda_i (\delta_{p_i} - \delta_{q_i}),
\end{equation*}
where $\lambda_i \in \R$,  $p_i, q_i \in \X$ and $N \in \N$. We use the shorthand notation $m_{p_i, q_i} = \delta_{p_i} - \delta_{q_i}$.
The \emph{Arens-Eells norm} of a molecule $m$ is
\begin{equation*}
    \|m\|_{\AEm(\X)}:= \inf \left\{ 
 \sum_{i=1}^N \lambda_i \sfd(p_i,q_i):\, m=\sum_{i=1}^N \lambda_i m_{p_i,q_i} \right\},
\end{equation*}
where the infimum is taken over all possible representations of $m$.
We define the \emph{Arens-Eells space} $\AEm(\X)$ as the completion of the space of molecules with respect to $\|\cdot\|_{\AEm}$. With some abuse of notation, we denote with the same symbol the norm on the completion.

For a distinguished point $0 \in \X$, we identify each molecule $m$ with the linear functional on $\Lip_0(\X)$ given by
\begin{equation*}
    m(\pi) = \int \pi \, \d m.
\end{equation*}
Following \cite[Chapter 3]{WeaverBook}, we see that $\AEm(\X)$ is isometric to a predual of $\Lip_0(\X)$ in the following sense.
For each $\pi \in \Lip_0(\X)$, there is the linear functional $I(\pi)$ on the space of molecules:
\begin{equation}\label{E:I-def}
    I(\pi)(m)=m(\pi).
\end{equation}
This functional is continuous and extends uniquely to a functional (which we denote with the same symbol)
\begin{equation*}
    I(\pi)\colon \AEm(\X) \to \R.
\end{equation*}
The map $I$ is an isometry of Banach spaces from $\Lip_0(\X)$ to $\AEm(\X)^*$. We shall omit writing out the map $I$ and identify $\Lip_0(\X)=\AEm(\X)^*$.

We denote by $\epsilon\colon \AEm(\X) \to (\Lip_0(\X))^*$ the canonical embedding into the second dual space. If $\X$ is separable, then so is $\AEm(\X)$ (and conversely) and therefore we may use the Kre\v{\i}n-Shmul'yan theorem (or rather, its immediate corollary \cite[Corollary 12.8]{Conway_fun_analysis}) to assert the following.

\begin{lemma}\label{lemma:KS_thm}
    If $\X$ is a separable metric space, then a linear functional $m \in \Lip_0(\X)^*$ is an element of $\epsilon(\AEm(\X))$ if and only if it is sequentially continuous with respect to pointwise convergence of bounded sequences.
\end{lemma}

In light of the previous lemma, we shall identify functionals in $\Lip_0(\X)^*$ sequentially continuous with respect to pointwise convergence of bounded sequences with the space $\AEm(\X)$.

\subsection{Metric currents}
\label{sec:preliminaries_currents}

For every $k \in \mathbb{N}$, we define $D^k(\X):= \Lip_b(\X)\times (\Lip(\X))^k$. The definition of metric currents was given in \cite{AK00} and is as follows.

\begin{definition}[Metric currents]

Let $(\X,\sfd)$ be a complete metric space and let $k \in \mathbb{N}$. We say that a multilinear functional $T\colon D^k(\X) \to \mathbb{R}$ is a \emph{metric $k$-current} provided it satisfies the following axioms:
\begin{itemize}
    \item(Locality) For every $(f,\pi_1,\dots,\pi_k) \in D^k(\X)$, we have $T(f,\pi_1,\dots,\pi_k)=0$ if there exists $i=1,\dots,k$ such that $\pi_i$ is constant on a neighborhood of $\{ f \neq 0 \}$.
    \item(Joint continuity) Given $f\in \Lip_b(\X)$ and $\pi_i^n,\pi_i \in \Lip(\X)$ such that
    $\sup_i\sup_{n \in \mathbb{N}} \Lip \pi_i^n < \infty$ and $\pi_i^n(x)\to \pi(x)$ for every $x \in \X$ and every $i$, we have
    \begin{equation*}
        \lim_{n \to \infty} T(f,\pi_1^n,\dots,\pi_k^n)= T(f,\pi_1,\dots,\pi_k).
    \end{equation*}
    \item(Finite-mass condition) There exists a nonnegative finite Borel measure such that
    \begin{equation}
        \label{eq:mass_bound_definition}
        \left|T(f,\pi_1,\dots,\pi_k)\right|\le \prod_{i=1}^k \LIP(\pi_i) \int |f|\,\d \mu.
    \end{equation}
\end{itemize}
The minimal measure $\mu$ that satisfies \eqref{eq:mass_bound_definition} is called the \emph{mass measure} of $T$ and is denoted by $\|T\|$.
\end{definition}

We call $\Mass(T):=\|T\|(\X)$ the mass of $T$. The vector space of all metric $k$-currents is a Banach space when endowed with the norm $\Mass(\cdot)$.

The \emph{boundary} of a metric $k$-current $T$ is the functional $\partial T\colon D^{k-1}(\X) \to \mathbb{R}$ defined as
\begin{equation*}
    \partial T(f,\pi_1,\dots,\pi_{k-1}):= T(1,f,\pi_1,\dots,\pi_k).
\end{equation*}

We say that a metric $k$-current is \emph{normal}, provided $\partial T$ is a metric $(k-1)$-current, or equivalently, if $\partial T$ has finite mass. We denote the vector space of all normal $k$-currents by $\cN_k(\X)$.

Let $(\X,\sfd_\X)$ and $(\Y,\sfd_\Y)$ be two complete metric spaces. For every Lipschitz map $\varphi \colon \X \to \Y$ and $T \in \cM_k(\X)$, we define a metric $k$-current $\varphi_* T$, called the \emph{pushforward} of $T$ via $\varphi$, as
\begin{equation*}
    \varphi_* T(f,\pi_1,\dots,\pi_k) = T(f\circ \varphi,\pi_1 \circ \varphi,\dots,\pi_k \circ \varphi).
\end{equation*}

Given a complete metric space $(\X,\sfd)$ and a closed set $C \subset \X$, $\cM_k(C)$ denotes the space of metric $k$-currents in the metric space $(C,\sfd\restr{C})$. There is an isometry
\begin{equation*}
    \Phi \colon \{T \in \cM_k(C):\, \|T\| \text{ is inner regular}\} \to \{ T \in \cM_k(X):\,{\rm supp}(T) \subset C,\,\|T\| \text{ is inner regular}  \},
\end{equation*}
where both spaces are endowed with the corresponding mass norm.
For a proof of this fact, see Proposition \ref{prop:isometry_of_currents_on_C} in the appendix of this paper.

Given a complete metric space $(\X,\sfd)$, $k \in \N$ and $T \in \cM_k(\X)$, we define the \emph{flat norm} of $T$ as
\begin{equation*}
\begin{aligned}
\Flatnorm(T):& = \inf \{\Mass(Q) + \Mass(\partial Q- T) :\, Q \in \cN_{k+1}(\X) \}\\
&=\inf \{\Mass(S) + \Mass(R) :\, S \in \cN_{k+1}(\X),\, R \in \cM_k(\X),\, T=\partial S+R \}.
\end{aligned}
\end{equation*}
In particular, notice that if $T \in \cN_k(\X)$, then, by the very definition of the flat norm,
\begin{equation*}
    \Flatnorm(T)=\inf \{\Mass(S) + \Mass(R) :\, S \in \cN_{k+1}(\X),\, R \in \cN_k(\X),\, T=\partial S+R \}.    
\end{equation*}

In the specific case of $(\mathbb{R}^2,|\cdot|)$, we introduce the following notation.
Given $a,b,c \in \R$ with $a < b$, we define $\curr{[a,b]\times \{c\}}$
\begin{equation*}
    \curr{[a,b]\times \{c\}}(f,\pi):= \int_a^b f(t,c) \frac{\d(\pi(t,c))}{\d t}\,\d t
\end{equation*}
and given $a,b,c \in \R$ with $b < c$ we define $\curr{\{a\}\times [b,c]}$ as 
\begin{equation*}
    \curr{\{a\}\times [b,c]}(f,\pi)=\int_b^c f(a,t) \frac{\d(\pi(a,t))}{\d t}\,\d t.
\end{equation*}

Notice that the last current, with the same definition, is a metric $1$-current also in the metric space $(\mathbb{R}^2,\sfd_\alpha)$, where $\sfd_\alpha((x,y),(x',y')):=\max\{|x-x'|^{\alpha},|y-y'|\}$ and $\alpha \in (0,1]$.

We will rely on the following representation result for normal $1$-currents in metric spaces, which was proved by Paolini and Stepanov.

\begin{theorem}[{\cite[Corollary 4.1]{PaoliniStepanov_cycles}}]\label{thm:Paolini_Stepanov}
Let $\X$ be a complete and separable metric space.
For $T\in\cM_1(\X)$ with $\partial T=0$ there is a Borel measure $\eta\in \mathcal{M}(C([0,1],\X))$
with total mass at most $\Mass(T)$,
concentrated over the Borel set $\Lip_1([0,1],\X)$, satisfying
\begin{align*}
T(f,\pi)&=\int\curr{\gamma}(f,\pi)\,\d\eta(\gamma) \\
\|T\|(B)&=\int\|\curr{\gamma}\|(B)\,\d\eta(\gamma)=\int\gamma_*\mathcal{L}^1(B)\,\d\eta(\gamma)
\end{align*}
for any Borel set $B\subset \X$.
\end{theorem}

\subsubsection{Pointwise estimates on metric currents}

The mass estimate in \eqref{eq:mass_bound_definition} for currents self-improves to a pointwise estimate with the pointwise Lipschitz constant.
These estimates will be useful for pointwise estimates in the homotopy formula in Section \ref{sec:deformations}.
\begin{proposition}
\label{prop:bound_current_*uppergradient}
Let $(\X,\sfd)$ be a complete and separable metric space, $k\in\N^+$, and $T \in \cM_k(\X)$.
Then
    \begin{equation}\label{eq:bound_current_local_Lip}
        |T(f,\pi_1,\dots,\pi_k)| \le \int |f|\,\prod_{i=1}^k\Lip \pi_i \,\d \|T\|,
    \end{equation}
    where $f\in\mathcal{B}^\infty(\X)$ and $\pi_1,\dots,\pi_k \in \Lip(\X)$.
\end{proposition}
\begin{proof}
\textsc{Step 1: Case $k=1$}.\par
First, observe that if $\pi\colon\X\to\R$ is a function and $L,r>0$, then
\begin{equation}\label{eq:bound_current_UG_1}
    S:=\{x\in\X\colon |\pi(x)-\pi(y)|\leq L \sfd(x,y), y\in B(x,r)\}
\end{equation}
is a closed set.
To see this, suppose $x\in\X$ is a limit point of $S$, $y\in B(x,r)$, and fix $0<\varepsilon<r$.
Then, we find $z\in S\cap B(x,\varepsilon)$ with $\sfd(x,z)<r-\sfd(x,y)$.
Since $\sfd(z,y)<r$ and $\sfd(z,x)<r$, the definition of $S$ implies
\begin{align*}
|\pi(x)-\pi(y)|&\leq |\pi(x)-\pi(z)|+|\pi(z)-\pi(y)|
    \leq L\varepsilon+L\sfd(z,y)\\
    &\leq 2L\varepsilon + L\sfd(x,y).
\end{align*}
Since $\varepsilon>0$ was arbitrary, this shows that $S$ is closed. \par
We now prove the claim.
Let $\pi\in\Lip(\X)$, fix $\varepsilon>0$, and let $\{q_i\}_{i\in\N^+}$ be a dense subset of the non-negative reals.
For $i,j\in\N^+$, set
\begin{equation*}
    S_{i,j}:=\{x\in\X\colon |\pi(x)-\pi(y)|\leq (q_i+\varepsilon)\sfd(x,y),y\in B(x,1/j),
    \text{ and } |\Lip \pi(x)-q_i|<\varepsilon\}.
\end{equation*}
Since $x\mapsto \Lip \pi(x)$ is Borel, the remark at the beginning of the proof
shows that $S_{i,j}$ is a Borel set.
Also, by definition of $\Lip\pi$, it is not difficult to see that the collection $S_{i,j}$ covers $\X$.
Since $\X$ is separable, we can cover each $S_{i,j}$ with countably many
Borel sets $S_{i,j,l}$ with $\mathrm{diam} (S_{i,j,l})<1/j$.
Observe that, for $x,y,z\in S_{i,j,l}$, it holds
\begin{equation*}
    |\pi(x)-\pi(y)|\leq (q_i+\varepsilon)\sfd(x,y)\leq (\Lip\pi(z)+2\varepsilon)\sfd(x,y).
\end{equation*}
Next, let $\{\tilde{E}_n\}_{n\in\N^+}$ be a re-indexing of $\{S_{i,j,l}\}_{i,j,l\in\N^+}$
and set $E_n:=\tilde{E}_n\setminus \cup_{m<n}\tilde{E}_m$.
Then, $\{E_n\}_{n \in \N^+}$ is a countable disjoint Borel cover of $\X$, satisfying
\begin{equation} \label{eq:bound_current_UG_2}
    \LIP(\pi\vert_{E_n})\leq \Lip\pi(x)+2\varepsilon,\qquad x\in E_n,n\in\N^+.
\end{equation}
Let $\pi_n\in\Lip(\X)$ be an extension of $\pi\vert_{E_n}$ having the same Lipschitz constant.
By the continuity and locality properties \cite[Theorem 3.5]{AK00} of metric currents we have
\begin{equation}
    \label{eq:locality_pi_pii}
        T(f,\pi)=\sum_{n=1}^\infty T(\chi_{E_n}f,\pi) = \sum_{n=1}^\infty T(\chi_{E_n}f,\pi_n).
\end{equation}
By \cref{eq:locality_pi_pii}, \cref{eq:bound_current_UG_2}, and the definition of $\pi_n$, we have
\begin{equation}
    \label{eq:main_computation_bound_current}
    \begin{aligned}
         |T(f,\pi)|
         &\leq\sum_{n=1}^\infty|T(\chi_{E_n}f,\pi_n)|
         \leq \sum_{n=1}^\infty\LIP(\pi_n)\int_{E_n}|f|\,\d\|T\| \\
         &\leq \sum_{n=1}^\infty\int_{E_n}(\Lip \pi+2\varepsilon)|f|\,\d\|T\| \\
         &\leq\int \Lip\pi|f|\,\d\|T\|+2\varepsilon\|f\|_\infty\|T\|(\X).
    \end{aligned}
\end{equation}
Since $\varepsilon>0$ was arbitrary, the above concludes the proof of the case $k=1$.\\
\textsc{Step 2: Case $k>1$.}\par
We proceed by induction on $k$.
We can assume $k>1$ and that \cref{eq:bound_current_local_Lip} holds for $k-1$.
Let $T\in\cM_k(\X)$ and, for $f\in\mathcal{B}^\infty(\X)$, $\pi_1,\dots,\pi_k\in\Lip(\X)$,
set
\begin{equation*}
    S_{\pi_2,\dots,\pi_k}(f,\pi_1):=R_{\pi_1}(f,\pi_2,\dots,\pi_k):=T(f,\pi_1,\dots,\pi_k).
\end{equation*}
Observe that $S_{\pi_2,\dots,\pi_k}\in\cM_1(\X)$, $R_{\pi_1}\in\cM_{k-1}(\X)$
and $\|R_{\pi_1}\|\leq\LIP(\pi_1)\|T\|$.
Hence, the induction hypothesis gives
\begin{equation*}
    |S_{\pi_2,\dots,\pi_k}(f,\pi_1)|=|R_{\pi_1}(f,\pi_2,\dots,\pi_k)|\leq\LIP(\pi_1)\int |f|\prod_{i=2}^k\Lip\pi_i\,\d\|T\|,
\end{equation*}
which implies
\begin{equation}\label{eq:bound_current_UG_3}
    \|S_{\pi_2,\dots,\pi_k}\|(B)\leq \int_B\prod_{i=2}^k\Lip\pi_i\,\d\|T\|,
\end{equation}
for $B\subset\X$ Borel.
Since $S_{\pi_2,\dots,\pi_k}$ is a metric $1$-current, the first step of the proof shows
\begin{equation*}
    |S_{\pi_2,\dots,\pi_k}(f,\pi_1)|\leq\int|f|\Lip\pi_1\,\d\|S_{\pi_2,\dots,\pi_k}\|.
\end{equation*}
The above inequality and \cref{eq:bound_current_UG_3} conclude the proof.
\end{proof}

Inequality \eqref{eq:bound_current_local_Lip} can be further improved, as we now describe.
Let $(\X,\sfd,\mu)$ be a complete and separable metric space endowed with a Radon 
measure which is finite on bounded sets.
In \cite{BateErikSoul2024}, it is shown that
it is possible to associate to each Lipschitz function $f\in\Lip(\X)$
a Borel function $|Df|_*$, called $*$-upper gradient, which controls the oscillation of $f$ on curve fragments, except for
a negligible set of curve fragments.
For a Lipschitz function $f\colon \X\to\R$, it holds
\begin{equation*}
    |Df|_*\leq \Lip f \qquad \mu\mathrm{-a.e.}\qquad \text{and}\qquad \Lip f \leq \Lip_a f \text{ everywhere}.
\end{equation*}
In \cite{BateErikSoul2024}, the following approximation result is proven.
\begin{theorem}[{\cite[Theorem 1.6]{BateErikSoul2024}}]
\label{thm:approximation_lemma_BES}
    Let $(\X,\sfd,\mu)$ be a complete and separable metric measure space. Let $f \in \Lip_{b}(\X)$ with bounded support. Then there exists a sequence $\{f_j\}_j \subset \Lip_{b}(\X)$ with bounded support
    with $\LIP(f_j) \le \LIP(f)$ and
    $|f_j| \le |f|$ for each $j \in \N$, sucht that 
    $f_j \to f$ pointwise everywhere (and so uniformly on compact sets)
    $\Lip_a\,f_j \to |Df |_*$ and $|Df_j|_* \to |Df|_*$ pointwise $\mu$-a.e.
\end{theorem}
A direct application of Theorem \ref{thm:approximation_lemma_BES}
and Proposition \ref{prop:bound_current_*uppergradient}
yields the following corollary.
\begin{corollary}
\label{cor:bound_*uppergradient}

Let $(\X,\sfd)$ be a complete and separable metric space, $k \in \N^+$ and $T \in \cM_k(\X)$. Then
\begin{equation}\label{eq:bound_current_*uppergradient}
    |T(f,\pi_1,\dots,\pi_k)|\leq \int|f|\prod_{i=1}^k|D\pi_i|_*\,\d\|T\|
\end{equation}
for $f\in\mathcal{B}^\infty(\X)$, and $\pi_1,\dots,\pi_k\in\Lip(\X)$,
where $|D\pi_i|_*$ is the $*$-upper gradient of $\pi_i$ in $(\X,\sfd,\|T\|)$.  
\end{corollary}

\begin{proof}
Fix $x_0\in\X$
and, for $R>0$, define $\varphi_R(x):=(2 -\frac{\sfd(x,x_0)}{R})^+\wedge 1$, $f^R:= f \varphi^R$ and $\pi_i^R:=\pi_i \varphi^{2R}$ for $i=1,\dots,k$. We have
\begin{equation}
\label{eq:locality_star_ug}    
|T(f^R,\pi_1,\dots,\pi_k)|=|T(f^R,\pi_1^R,\dots,\pi_k^R)| \le \int|f^R|\prod_{i=1}^k|D\pi^R_i|_*\,\d\|T\| = \int|f^R|\prod_{i=1}^k|D\pi_i|_*\,\d\|T\|, 
\end{equation}
where we used the locality axiom of currents, Proposition \ref{prop:bound_current_*uppergradient},
and Theorem \ref{thm:approximation_lemma_BES} applied
to the functions $\{\pi_i^R\}_i$.
In particular, the last equality follows from \cite[Lemma 2.12, item (2)]{BateErikSoul2024}.
By letting $R$ go to infinity in \eqref{eq:locality_star_ug}, we conclude by applying the continuity axiom of currents and monotone convergence.
\end{proof}

\begin{remark}
Let $K\subset [0,1]$ be a fat Cantor set and consider the metric space $(K,|\cdot|)$. Let $T\in\cM_1(K)$ be the current induced by $K$, i.e.\ $T(f,\pi):=\int_K f(t) \pi'(t)\,\d t$.
The set $K$ does not contain any non-constant rectifiable curve. Therefore, for any $p\in [1,\infty]$, every Borel function on $K$ has $0$ as minimal
$p$-weak upper gradient in the sense of Heinonen-Koskela (for definitions, see \cite[Section 7]{HKST2015}).
It follows that, in general, the $*$-upper gradients in \cref{eq:bound_current_*uppergradient} cannot
be replaced with any weak upper gradient, for any $p \in [1,\infty]$.
\end{remark}

\section{Flat chain conjecture for metric \texorpdfstring{$1$}{1}-currents, the isomorphism theorem and quasiconvexity}
\label{sec:flat_chain_conjecture}

To avoid trivial cases, throughout this section we assume that all metric spaces have at least two points.
\begin{lemma}\label{lemma:bdry_operator_well_defined}
Let $(\X,\sfd)$ be a complete and separable metric space.
Then $\partial T\in\AEm(\X)$ for $T\in\cM_1(\X)$,
hence the boundary $\partial$ is a linear map $\partial \colon\cM_1(\X)\to\AEm(\X)$.
Moreover, $\|\partial\|=1$ if $\X$ is not purely $1$-unrectifiable, while $\cM_1(\X)=\{0\}$ otherwise.
\end{lemma}
\begin{proof}
From the definition of metric current, it follows that $\partial T\colon \mathrm{Lip}_0(X)\to\R$ is linear and weak* sequentially continuous and hence weak* continuous, by
\cref{lemma:KS_thm}.
It follows that the boundary operator maps $\cM_1(\X)$ to $\AEm(\X)$. \par
From $|\partial T(f)|\leq \LIP(f)\Mass{(T)}$ for $f \in \Lip_0(\X)$, we have $\|\partial T\|_{\AEm}\leq \Mass{(T)}$
and so $\|\partial\|\leq 1$.
If $X$ is not purely 1-unrectifiable, there are a compact set $K\subset\R$ of positive measure and a biLipschitz map $\gamma\colon K\to \X$.
Almost every point $t\in K$ is a density point of $K$, a Lebesgue point of $t\mapsto|\dot{\gamma}_t|$, and the metric derivative $|\dot{\gamma}_t|$ exists. Let $t_0$ be any such point and choose $t_n\downarrow t_0$ with $t_n\in K$, and set
$\gamma_n:=\gamma\vert_{[t_0,t_n]}$.

The conclusion follows if we prove
\begin{equation*}
    \liminf_{n \to \infty} \frac{\|\partial \curr{\gamma_n}\|_{\AEm(\X)}}{\Mass(\curr{\gamma_n})} \ge 1.
\end{equation*}
We compute for $f \in \Lip_0(\X)$ with $\LIP(f) \le 1$
\begin{equation*}
\begin{aligned}
    \partial \curr{\gamma_n}(f) &= \int_{[t_0,t_n] \cap K} (f \circ \gamma)'_t\,\d t \ge \int_{[t_0,t_n]} g'\,\d t- \LIP(\gamma) \mathcal{L}^1([t_0,t_n]\setminus K)\\
    &  = (f\circ\gamma)_{t_n}-(f\circ\gamma)_{t_0}- \LIP(\gamma) \mathcal{L}^1([t_0,t_n]\setminus K),
\end{aligned}
\end{equation*}
where $g_n \colon \mathbb{R}\to \mathbb{R}$ is the MacShane Lipschitz extension of $f \circ \gamma$ to $\mathbb{R}$ and we used $\LIP(g_n)\leq \LIP(\gamma_n) \le \LIP(\gamma)$. By choosing $f = \sfd((\gamma_n)_{t_0},\cdot)-\sfd((\gamma_n)_{t_0},0) \in \Lip_0(\X)$, we have
\begin{equation*}
    \|\partial \curr{\gamma_n}\|_{\AEm(\X)} \ge \sfd((\gamma_n)_{t_n}, (\gamma_n)_{t_0})-\Lip(\gamma) |[t_0,t_n]\setminus K|.
\end{equation*}
Note that $\Mass{(\curr{\gamma_n})}=\int_{[t_0,t_n]\cap K}|\dot{\gamma}_t|\,\d t$ and so $\Mass{(\curr{\gamma_n})}/(t_n-t_0)\rightarrow |\dot\gamma_{t_0}|$. Finally,
by our choice of $t_0$ and the above estimate, we have
\begin{equation*}
    \liminf_{n\rightarrow \infty}\frac{\|\partial\curr{\gamma_n}\|_{\AEm(\X)}}{\Mass(\curr{\gamma_n})}\geq \lim_{n\rightarrow\infty}\frac{d(\gamma_{t_0},\gamma_{t_n})}{|\dot\gamma_{t_0}|(t_n-t_0)}-\LIP(\gamma)\lim_{n\rightarrow\infty}\frac{|[t_n,t_0]\setminus K|}{|\dot\gamma_{t_0}|(t_n-t_0)}=1.
\end{equation*}
This shows that $\|\partial\|=1$. \par
Suppose now $\X$ is purely $1$-unrectifiable.
Then, for every biLipschitz $\gamma\in\Gamma(\X)$
and Lipschitz $f\colon\X\to\R$ it holds $(f\circ\gamma)_t'=0$ for a.e.\ $t\in{\rm dom}(\gamma)$.
That is, in the terminology of \cite{BateErikSoul2024},
every Lipschitz function $f\colon\X\to\R$ has $0$ as $*$-upper gradient;
see the comment after \cite[Proposition 2.10]{BateErikSoul2024}.
Hence, from Corollary \ref{cor:bound_*uppergradient}, we deduce $\cM_1(\X)=\{0\}$
and thus $\partial =0$.
\end{proof}
By \cref{lemma:bdry_operator_well_defined}, we know that $\partial\colon \cM_1(\X)\to \AEm(\X)$ is a bounded linear operator. It therefore induces a bounded and injective operator on the quotient Banach space $(\cM_1(\X)/\mathrm{ker}(\partial),\|\cdot\|)$, which we still denote with $\partial$.
Given a metric space $(\X,\sfd)$, define
\begin{equation*}
    \mathrm{qc}(\X):=\inf\{C\in [1,\infty]\colon \X \text{ is } C\text{-quasiconvex}\}.
\end{equation*}
\begin{theorem}[Isomorphism theorem]
\label{thm:isomorphism}
Let $(\X,\sfd)$ be a complete and separable metric space.
Then 
\[\partial\colon \cM_1(\X)/\mathrm{ker}(\partial)\to \AEm(\X)\]
is a Banach space isomorphism if and only if $\X$ is quasiconvex.
In the latter case, we have
\begin{equation}
\label{eq:3.2inThm}
    \mathrm{qc}(X)^{-1}\|\partial T\|_{\AEm(\X)}\leq \|[T]\|\leq \|\partial T\|_{\AEm(\X)} \quad\text{for every $T\in\cM_1(\X)$.}
\end{equation}
with the constants in \eqref{eq:3.2inThm} being optimal.
\end{theorem}
\begin{anfxnote}{Comparison with other work}
The above can be seen as a generalisation results of the type https://arxiv.org/pdf/2503.04390
We should check if we can recover the results in the paper from the isomorphism theorem
\end{anfxnote}
The optimality of the constants can be equivalently formulated as $\|\partial\|=1$ and $\|\partial^{-1}\|=\mathrm{qc}(\X)$.
Given a metric space $\X$ and two points $x,y\in \X$, we set
\begin{equation*}
    \mathrm{qc}(x,y):=\inf\{C\in [1,\infty]\colon \text{there is a }C\text{-quasiconvex curve from }x\text{ to }y\},
\end{equation*}
with the usual convention $\mathrm{qc}(x,y)=\infty$ if no such curve exists.
Note that $\mathrm{qc}(\X)=\sup_{x, y}\mathrm{qc}(x,y)$.

\begin{lemma}
\label{lemma:isomorphism_and_quasiconvexity}
    Let $(\X,\sfd)$ be a complete and separable metric space. Let $m \in \AEm(\X)$ be a molecule. Then, for any $T \in \cM_1(\X)$ with $\partial T=m$, it holds \begin{equation}
    \label{eq:bound_from_below_mass}
        \Mass(T) \ge \|m\|_{\AEm(\X,\sfd_\ell)},
    \end{equation}
    where $\sfd_\ell\colon\X\times\X\to[0,\infty]$ is the (generalised) length distance induced by $\sfd$.
    In particular, if $m:=\delta_y-\delta_x\in\AEm(X)$ with $x,y \in \X$, we have for every $T \in \cM_1(\X)$ with $\partial T=m$ that $\Mass(T)\geq\mathrm{qc}(x,y)\sfd(x,y)$.
\end{lemma}

\begin{proof}
    We assume $m\neq 0$, otherwise there is nothing to prove.
    Since $m \in \AEm(\X)$ is a molecule, we write it as $m=\sum_{i=1}^n a_i (\delta_{y_i}-\delta_{x_i})$ for some $n \in \mathbb{N}$, where $a_i >0$ for every $i$ and the points $x_i,y_i$ are all distinct.
    Since $T$ is a normal $1$-current, we apply \cite[Proposition 3.8]{PaoliniStepanov_acyclic} and we have that $T=C+T_0$, where $C$ is a cycle of $T$ and $T_0$ is acyclic, $\Mass(T_0) \le \Mass(T)$ and $\partial T_0 =m$; see \cite[Definition 3.7]{PaoliniStepanov_acyclic} for the definition of cycle and acyclic current. We apply \cite[Theorem 5.1]{PaoliniStepanov_acyclic} to get a nonnegative Borel measure $\eta$ on $C([0,1],\X)$ concentrated on rectifiable paths such that
    \begin{equation}
    \label{eq:representation_of_T0}
        T_0 = \int \curr{\gamma}\,\d \eta(\gamma),\quad \Mass(T_0)= \int \ell(\gamma)\,\d \eta(\gamma)
    \end{equation}
    and 
    \begin{equation}
    \label{eq:representation_of_T0_2}
     (e_0)_*\eta =(\partial T_0)^-=\sum_i a_i\delta_{x_i}\text{,}\quad(e_1)_*\eta = (\partial T_0)^+=\sum_i a_i \delta_{y_i}.   
    \end{equation}
    For $1\leq i,j\leq n$, we define
    \begin{equation*}
        \eta_{x_i,y_j}:=\eta\restr{e_0^{-1}\{x_i\}\cap e_1^{-1}\{y_j\}}
    \end{equation*}
    and $b_{i,j}:=\mathbb{\eta}_{x_i,y_j}(C([0,1],\X))$.

    Equation \eqref{eq:representation_of_T0_2} implies ${\e_0}_*\eta(\X\setminus \{x_i\}_i)={\e_1}_*\eta(\X\setminus \{y_i\}_i)=0$, from which $\eta(e_1^{-1}(\X \setminus \{y_i\}_i) \cap e_0^{-1}(\X \setminus \{x_i\}_i))=0$ follows. That is, $\eta$ is concentrated on
    \begin{equation*}
        \{\gamma \in C([0,1],\X):\, \gamma_0 \in \{x_1,\dots,x_n\},\,\gamma_1 \in \{y_1,\dots,y_n\} \}.
    \end{equation*}
    This property and the fact that $\{\gamma:\gamma_0=x_i\}=\cup_j \{\gamma:\gamma_0=x_i,\gamma_0=y_j\}$ where the sets in the union are disjoint imply
    \begin{equation}
        \label{eq:balance_conditions}
        \begin{aligned}
        &\sum_{i,j=1}^n \eta_{x_i,y_j} = \eta,\\
        &\sum_i b_{i,j}=a_j\text{ for every }j \in \{1,\dots,n\}\text{ and }\sum_j b_{i,j}=a_i\text{ for every }i \in \{1,\dots,n\}.
        \end{aligned}
    \end{equation}
    Moreover, since $\eta_{x_i,y_j}$ is concentrated on curves connecting $x_i$ to $y_j$ we have that for every $i,j$
    \begin{equation}
    \label{eq:bound_intrinsic_distance}
        \ell(\gamma)\ge\sfd_\ell(x_i,y_j) \quad \text{for }\eta_{x_i,y_j}\text{-a.e.\ }\gamma.
    \end{equation}

    Let $f \in \Lip(\X,\sfd_\ell)$ be $1$-Lipschitz.
    We compute
    \begin{equation*}
    \begin{aligned}
        \sum_{i=1}^m a_i \left( \int f \,\d \delta_{x_i}- \int f \,\d \delta_{y_i}
    \right) &\stackrel{\eqref{eq:balance_conditions}}{=}\sum_{i,j=1}^n b_{i,j} (f(x_i)-f(y_j)) \le \sum_{i=1}^n \sum_{j=1}^n b_{i,j} \sfd_\ell(x_i,y_j)\\
    & \stackrel{\eqref{eq:bound_intrinsic_distance}}{\le} \sum_{i=1}^n \sum_{j=1}^n \int \ell(\gamma) \,\d \eta_{x_i,y_j} \stackrel{\eqref{eq:balance_conditions}}{=} \int \ell(\gamma)\,\d \eta\stackrel{\eqref{eq:representation_of_T0}}{=}\Mass(T_0) \le \Mass(T).
    \end{aligned}
    \end{equation*}
    By taking the supremum over $1$-Lipschitz functions we have \eqref{eq:bound_from_below_mass}.
    Lastly, note that $\sfd_\ell(x,y)=\mathrm{qc}(x,y)\sfd(x,y)$ for $x,y\in\X$.
\end{proof}

\begin{proof}[Proof of Theorem \ref{thm:isomorphism}]
From Lemma \ref{lemma:bdry_operator_well_defined}, the linear operator
$\partial\colon \cM_1(X)/\mathrm{ker}(\partial)\to \AEm(X)$ is well-defined and bounded.
Assume $X$ is quasiconvex and consider the adjoint
$\partial^*\colon \mathrm{Lip}_0(X)\to (\cM_1(X)/\mathrm{ker}(\partial))^*$,
where we have used the isometric identification of $(\AEm(X))^*$ with $\mathrm{Lip}_0(X)$ from \eqref{E:I-def}.

The operator
\begin{equation*}
    \partial\colon \cM_1(\X)/\mathrm{ker}(\partial)\to \AEm(\X)
\end{equation*}
is injective by definition. It follows that the adjoint operator $\partial^*$ has weak$^*$ dense range; see \cite[Chapter VI, Proposition 1.8]{Conway_fun_analysis} together with the bipolar theorem \cite[1.8 Bipolar Theorem, Chapter 5]{Conway_fun_analysis}.

We claim that, for $f\in\mathrm{Lip}_0(\X)$, it holds 
\begin{equation}\label{eq:isomorphism_closed_range}
    \|\partial^* f\|\geq \mathrm{qc}(\X)^{-1}\mathrm{Lip}(f).
\end{equation}
Assuming the claim, this implies that $\partial^*$ is injective and has norm-closed range.
Being an adjoint operator, $\partial^*$ has norm-closed range if and only if it has weak*-closed range \cite[Theorem 1.10]{Conway_fun_analysis} which, by the previous discussion, is the case if and only if $\partial^*$ is onto.
Therefore, we have that $\partial^*$ is an isomorphism,
and hence so is $\partial$ \cite[Chapter VI, Proposition 1.9]{Conway_fun_analysis}. \par
Let $f\in\mathrm{Lip}_0(X)$ be nonzero, $0<\varepsilon<\LIP(f)$, and $x,y\in X$ distinct such that
$f(y)-f(x)\geq (\LIP(f)-\varepsilon)d(x,y)$.
Let $\gamma\in C([0,1];X)$ be a $(\mathrm{qc}(X)+\varepsilon)$-quasiconvex curve from $x$ to $y$ and set $T_0:=\curr{\gamma}/\ell(\gamma)$. Note that $\|[T_0]\|\leq \Mass{(T_0)}\leq 1$.
We have
\begin{align*}
\|\partial^*f\| &= \sup\{\langle\partial^*f,[T]\rangle\colon \|[T]\|\leq 1\}
= \sup\{\langle f,\partial T\rangle\colon \|[T]\|\leq 1\} \\
&\geq \partial T_0(f) =\frac{(f\circ\gamma)_1-(f\circ\gamma)_0}{\ell(\gamma)}\geq \frac{\LIP(f)-\varepsilon}{\mathrm{qc}(\X)+\varepsilon}.
\end{align*}
Since $0<\varepsilon<\LIP(f)$ was arbitrary, we conclude that $\|\partial^*f\|\geq \mathrm{qc}(\X)^{-1}\LIP(f)$ for each $f\in\Lip_0(\X)$.
This concludes the proof of the claim.

Moreover, we have
$\|\partial^{-1}\|=\|(\partial^*)^{-1}\|\leq \mathrm{qc}(\X)$, thus together with \cref{lemma:isomorphism_and_quasiconvexity} gives $\|\partial^{-1}\|=\mathrm{qc}(X)$. \par
Lastly, if $X$ is not quasiconvex, we prove that $\partial \colon \cM_1(\X)/\ker(\partial) \to \AEm(\X)$ is not an isomorphism.
Assume it is. Since $\X$ is 
not quasiconvex, there exist $x_i,y_i \in \X$ such that $\mathrm{qc}(x_i,y_i) \ge i$ for every $i \in \mathbb{N}$. Setting $m_i:=\delta_{y_i}-\delta_{x_i} \in \AEm(\X)$, by Lemma \ref{lemma:isomorphism_and_quasiconvexity} we have that for every $T \in \cM_1(\X)$ with $\partial T=m_i$ we have $\Mass{(T)}\geq i\d(x_i,y_i)=i\|m_i\|_{\AEm(\X)}$.
Hence,
\begin{equation*}
\|\partial^{-1}m_i\|\geq i\|m_i\|_{\AEm(\X)},
\end{equation*}
contradicting the boundedness assumption on $\partial^{-1}$.
\end{proof}
\begin{corollary}
\label{coro:approximation_by_normal_currents}
Let $\X$ be a complete and quasiconvex metric space.
Let $T\in\cM_1(\X)$ and suppose its mass measure is inner regular by compact sets.
Then there is a sequence $T_i$ of normal $1$-currents such that
$\Mass{(T-T_i)}\rightarrow 0$.
\end{corollary}
\begin{proof}
We first reduce to the separable case.
Since $\|T\|$ is inner regular by compact sets, it is concentrated on a $\sigma$-compact
set $S$, which in particular is separable.
By \cite[Lemma 2.1]{Hajlasz2018}, there is a closed, separable, quasiconvex
set $F$ containing $S$.
Applying \cref{lemma:restrictions_of_currents}, we find a metric current $\widehat{T}\in\cM_1(F)$
such that $\iota_*\widehat{T}=T$, where $\iota\colon F\to\X$ is the inclusion map.\par
Recall that molecules are dense in $\AEm(F)$ and so there exists a sequence $m_i\in\AEm(F)$ of molecules converging to $\partial \widehat{T}\in\AEm(F)$ in the Arens-Eells norm.
Since $F$ satisfies the assumptions of \cref{thm:isomorphism}, we have $\|\partial^{-1}m_i-[\widehat{T}]\|\rightarrow 0$ as $i\rightarrow \infty$.
The definition of quotient norm reads
\begin{equation*}
    \|\partial^{-1}m_i-[\widehat{T}]\|=\inf\{\Mass{(\partial^{-1}m_i+C-\widehat{T})}\colon C\in\cM_1(F),\partial C=0\},
\end{equation*}
and so there are $C_i\in\cM_1(F)$ such that $\Mass{(\partial^{-1}m_i+C_i-\widehat{T})}\rightarrow 0$.
Then $\widehat{T}_i:=\partial^{-1}m_i+C_i$ is metric $1$-current whose boundary $\partial \widehat{T}_i=m_i$ is a molecule, i.e.\ a linear combination of Dirac deltas. In particular, $\widehat{T}_i$ and $T_i:=\iota_*\widehat{T}_i$ are normal currents.
Lastly, since $\iota$ is $1$-Lipschitz
\begin{equation*}
    \Mass(T_i-T)=\Mass(\iota_*(\widehat{T}_i-\widehat{T}))\leq\Mass(\widehat{T}_i-\widehat{T})\rightarrow 0.
\end{equation*}
\end{proof}

\section{Homotopy formulas and polyhedral approximation theorems for normal currents}
\label{sec:deformations}

The goal of this section is to prove a strict polyhedral approximation theorem for normal $1$-currents in metric spaces admitting a conical geodesic bicombing. This class of metric spaces includes Banach spaces and ${\sf CAT}(0)$ spaces. The main tool to prove this result is a homotopy formula in the same setting.
In the last part of the section, we prove a homotopy formula for normal $k$-currents in Banach spaces for $k\geq 1$. This last result will be employed in Section \ref{sec:alberti_marchese}.

\subsection{Strict polyhedral approximation theorem for \texorpdfstring{$k=1$}{k=1} in metric spaces admitting a conical geodesic bicombing}

A \emph{geodesic bicombing} on a geodesic metric space $(\X,\sfd)$ is a map $\sigma\colon \X\times\X\times[0,1]\to\X$ such that, for $x,y\in\X$, $\sigma_{x,y}:=\sigma(x,y,\cdot)$ is a
(constant speed) geodesic from $x$ to $y$. It is \emph{conical} if it additionally satisfies
\begin{equation*}
    \sfd(\sigma_{x,y}(t),\sigma_{x',y'}(t)) \le (1-t)\sfd(x,x')+t\sfd(y,y'),\qquad t \in [0,1] 
\end{equation*}
and for all $x,y,x',y' \in \X$.

This condition was proven in injective metric spaces in \cite[Proposition 3.8]{Lang2013} and then deeply studied in \cite{DescombesLang2015}; it includes convex subsets of 
Banach spaces (where the bicombing is given by line segments) and ${\sf CAT}(0)$ spaces (where the bicombing is given by the unique geodesic connecting two points).

\begin{lemma}[Simplified homotopy lemma in conical geodesic bicombing metric spaces]
\label{lemma:homotopy_lemma}
Let $(\X,\sfd)$ be a complete metric space admitting a conical geodesic bicombing.
Then, for every two Lipschitz curves $\gamma^0,\gamma^1\colon[0,1]\to\X$,
there are metric currents $R=R_{\gamma^0,\gamma^1}\in\cM_1(\X)$, $S=S_{\gamma^0,\gamma^1}\in\cM_2(\X)$ satisfying
$\curr{\gamma^0}-\curr{\gamma^1}=\partial S + R$,
and
    \begin{align*}
        \Mass(S) &\leq \big(\ell(\gamma^0)+\ell(\gamma^1)\big)\sfd_{\infty}(\gamma^0, \gamma^1), \\
        \Mass(R)&\leq \sfd(\gamma^0_0,\gamma^1_0)+\sfd(\gamma^0_1,\gamma^1_1)\leq 2\sfd_{\infty}(\gamma^0, \gamma^1);
    \end{align*}
in particular, 
    \begin{equation*}
        \Flatnorm(\curr{\gamma^0}-\curr{\gamma^1}) \le \big(\ell(\gamma^0)+\ell(\gamma^1)+2\big)\sfd_{\infty}(\gamma^0, \gamma^1).
    \end{equation*}
Moreover, for each $f\in\mathcal{B}^\infty(\X)$, $\pi_1,\pi_2,\pi\in\Lip(\X)$,
the maps
\begin{align*}
    (\gamma^0,\gamma^1)&\mapsto S_{\gamma^0,\gamma^1}(f,\pi_1,\pi_2),\\
    (\gamma^0,\gamma^1)&\mapsto R_{\gamma^0,\gamma^1}(f,\pi)
\end{align*}
are Borel measurable. Here, recall that $\Lip([0,1];\X)$ is endowed with the topology induced by the supremum distance $\sfd_\infty$.
\end{lemma}

We recall the notation in Section \ref{sec:preliminaries_currents} for $\curr{[a,b]\times \{c\}}$ and $\curr{\{a\}\times [b,c]}$.

\begin{proof} 
Let $\sigma$ denote a conical geodesic bicombing on $\X$ and define $H\colon \Lip([0,1];\X)^2\times [0,1]^2\to\X$,
\begin{equation}\label{eq:homotopy}
        H(s,t):=H_{\gamma^0,\gamma^1}(s,t):=\sigma(\gamma^0_s,\gamma^1_s,t).
\end{equation}
We will obtain the currents $S$ and $R$ from the pushforward under $H$ of $e_1\wedge e_2[0,1]^2$; see \cref{eq:square_current}.
We start by showing that $H$ is Lipschitz, with constant
$\sqrt{2}L$, where
$$L:=\max\{\LIP(\gamma^0),\LIP(\gamma^1),\sfd_\infty(\gamma^0,\gamma^1)\}.$$
   Recall that by definition of conical geodesic bicombing and \cref{eq:homotopy}, we have
     \begin{equation}
     \label{eq:conical_geo_bicombing_}
     \begin{aligned}
         \sfd(H(s,t),H(s',t))& \le (1-t)\sfd(\gamma^0_{s},\gamma^0_{s'})+t\sfd(\gamma^1_s,\gamma^1_{s'}) \\
         \sfd(H(s,t),H(s,t'))&=|t-t'|\sfd(\gamma^0_{s},\gamma^1_{s})
     \end{aligned}
     \end{equation}
     for all $t,t',s,s'\in[0,1]$.
     Hence,
     \begin{equation*}
    \begin{aligned}
        \sfd(H(s,t),H(s',t')) &\le \sfd(H(s,t),H(s',t))+\sfd(H(s',t),H(s',t'))\\ 
        &\stackrel{\eqref{eq:conical_geo_bicombing_}}{\le} (1-t)\sfd(\gamma^0_s,\gamma^0_{s'})+ t\sfd(\gamma^1_{s},\gamma^1_{s'}) + |t-t'|\sfd_\infty(\gamma^0,\gamma^1)\\ 
        & \le L|s-s'| + L|t-t'|\leq \sqrt{2}L \|(s,t)-(s',t')\|_2.
    \end{aligned}
    \end{equation*}
We define $S:=H_*(e_1 \wedge e_2 [0,1]^2) \in \cM_2(\X)$,
where $e_1 \wedge e_2 [0,1]^2 \in \cM_2(\mathbb{R}^2)$ is given by 
    \begin{equation}\label{eq:square_current}
        (e_1 \wedge e_2 [0,1]^2)(f,\pi_1,\pi_2):=\int_{[0,1]^2} f(\partial_1 \pi_1 \partial_2 \pi_2- \partial_2 \pi_1 \partial_1 \pi_2)\,\d \mathcal{L}^2.
    \end{equation}
Testing on smooth functions and then mollifying, we see that
\begin{equation*}
    \partial(e_1 \wedge e_2 [0,1]^2)= \curr{[0,1]\times \{0\}}+ \curr{\{1\}\times [0,1]} - \curr{[0,1]\times \{1\}} - \curr{\{0\}\times [0,1]}.
\end{equation*}
Therefore,
\begin{equation*}
\begin{aligned}
    \partial S &=\partial H_*(e_1 \wedge e_2 [0,1]^2) = H_*\partial (e_1 \wedge e_2 [0,1]^2) \\
    &= H_*\curr{[0,1]\times \{0\}}+ H_*\curr{\{1\}\times [0,1]} - H_*\curr{[0,1]\times \{1\}} - H_*\curr{\{0\}\times [0,1]}\\
    & = \curr{\gamma^0} + \curr{H(1,\cdot)} - \curr{\gamma^1}- \curr{H(0,\cdot)}=: \curr{\gamma^0} - \curr{\gamma^1} -R.
\end{aligned}
\end{equation*}
We now estimate the masses of $R$ and $S$, starting from $R$.
Recall that, for each $s\in[0,1]$, $H(s,\cdot)$ is a geodesic from $\gamma^0_s$ to $\gamma^1_s$. There follows $\Mass{(\curr{H(s,\cdot)})}\leq \sfd(\gamma^0_s,\gamma^1_s)$ and
\begin{equation*}
\Mass{(R)}\leq \sfd(\gamma^0_0,\gamma^1_0)+\sfd(\gamma^0_1,\gamma^1_1).
\end{equation*}
We now move on to $S$.
For each $t\in[0,1]$, the function $s\in[0,1]\mapsto H(s,t)$ is Lipschitz and therefore
admits a metric derivative at $\mathcal{L}^1$-a.e.\ point $s$, which we denote with $|\partial_1H|(s,t)$. We similarly define $|\partial_2H|(s,t)$.
From \cref{eq:conical_geo_bicombing_}
\begin{align*}
|\partial_1H|(s,t)&\leq (1-t)|\dot{\gamma^0_s}|+t|\dot{\gamma^1_s}|, \\
|\partial_2H|(s,t)&=\sfd(\gamma^0_s,\gamma^1_s)\leq \sfd_\infty(\gamma^0,\gamma^1)    
\end{align*}
at $\mathcal{L}^2$-a.e.\ $(s,t)\in[0,1]^2$.
If $B\subset\X$ is Borel and $\pi_1,\pi_2\colon\X\to\R$ $1$-Lipschitz, we have
\begin{equation*}
\begin{aligned}
    |S(\chi_{B},\pi_1,\pi_2)| &\le \int_{[0,1]^2} \chi_{B}\circ H\left|\partial_1 (\pi_1 \circ H) \partial_2 (\pi_2 \circ H) - \partial_2 (\pi_1 \circ H) \partial_1 (\pi_2 \circ H)\right|\,\d \mathcal{L}^2\\
    &\le 2\int_{[0,1]^2} \chi_B\circ H |\partial_1 H| |\partial_2 H|\d\mathcal{L}^2\\
    &\le 2 \sfd_\infty(\gamma^0,\gamma^1)\,
        \int_{[0,1]^2} \chi_B \circ H(s,t)\big((1-t)|\dot{\gamma^0_s}|+t|\dot{\gamma^1_s}|\big)\,\d s\d t.
\end{aligned}
\end{equation*}
Hence, for any Borel partition of $\X$ $\{B_i\}_i$ and $1$-Lipschitz maps $\pi^i_1,\pi^i_2\colon\X\to\R$, it holds
\begin{align*}
\sum_i|S(\chi_{B_i},\pi^i_1,\pi^i_2)|&\leq2\sfd_\infty(\gamma^0,\gamma^1)
\sum_i\int_{[0,1]^2} |\chi_{B_i} \circ H(s,t)|\big((1-t)|\dot{\gamma^0_s}|+t|\dot{\gamma^1_s}|\big)\,\d s\d t \\
&=2\sfd_\infty(\gamma^0,\gamma^1)\int_{[0,1]^2}(1-t)|\dot{\gamma^0_s}|+t|\dot{\gamma^1_s}|\,\d s\d t \\
&=\sfd_\infty(\gamma^0,\gamma^1)\big(\ell(\gamma^0)+\ell(\gamma^1)\big),
\end{align*}
which proves $\Mass(S)\leq \sfd_\infty(\gamma^0,\gamma^1)\big(\ell(\gamma^0)+\ell(\gamma^1)\big)$.
It remains to show measurabily of $R_{\gamma^0,\gamma^1}$ and $S_{\gamma^0,\gamma^1}$.
Conical geodesic bicombings are always continuous \cite{DescombesLang2015}.
Since the evaluation map $(t,\gamma)\in[0,1]\times\Lip([0,1];\X)\mapsto \gamma_t\in\X$ is continuous, it follows that
$(s,t,\gamma^0,\gamma^1)\mapsto \sigma(t,\gamma^0_s,\gamma^1_s)$ is also continuous.
Then, $(s,t,\gamma^0,\gamma^1)\mapsto f\circ H_{\gamma^0,\gamma^1}(s,t)$ and $(s,t,\gamma^0,\gamma^1)\mapsto \partial_i(\pi_j\circ H_{\gamma^0,\gamma^1})(s,t)$
are Borel.
Finally, Fubini's theorem allows us to conclude.
\end{proof}

\begin{remark}[Extension to ${\sf CAT}(\kappa)$ spaces]
    We have already noted that ${\sf CAT}(0)$ spaces admit conical geodesic bicombings. 
    In fact, a statement analogous to Lemma \ref{lemma:homotopy_lemma} holds in ${\sf CAT}(\kappa)$ spaces with $\kappa >0$, provided we further assume
    ${\rm im}(\gamma^i) \subset B(z,\frac{\pi}{2\sqrt{\kappa}})$ for some $z \in \X$ and all $i=0,1$.
    To see this, let $H\colon[0,1]\times[0,1]\to\X$ be such that
    $t\mapsto H(s,t)$ is the (unique) geodesic from $\gamma^0_s$ to $\gamma^1_s$
    and observe that, by convexity of $B(z,\frac{\pi}{2\sqrt{\kappa}})$, 
    we have $H(s,t)\in B(z,\frac{\pi}{2\sqrt{\kappa}})$
    for $t\in [0,1]$.
    Then,
    combining \cite[Lemma 2.1]{CavallucciSambusetti} with the argument of \cite[Proposition II.2.2]{BridsonHaefliger}, we get
    \begin{equation}\label{eq:bicombing_CAT_kappa}
    \sfd(H(s,t),H(s',t)) \le 2(1-t)\sfd(\gamma^0_{s},\gamma^0_{s'})+2t\sfd(\gamma^1_s,\gamma^1_{s'}),\qquad\text{for }t \in [0,1].
    \end{equation}
    One can then repeat the proof of Lemma \ref{lemma:homotopy_lemma}
    with \eqref{eq:bicombing_CAT_kappa} in place of the
    first equation in \eqref{eq:conical_geo_bicombing_}
    to obtain
    \begin{equation*}
        \Flatnorm(\curr{\gamma^0}-\curr{\gamma^1}) \le 2(\ell(\gamma^0)+\ell(\gamma^1)+2)\sfd_\infty(\gamma^0,\gamma^1).
    \end{equation*}
\end{remark}

\begin{remark}
We note that it is necessary to assume some connectivity of the metric space in order to have the continuity of the map
\begin{equation*}
    \Lip([0,1],\X) \ni \gamma \mapsto \curr{\gamma} \in \cM_1(\X)
\end{equation*}
on sets of equibounded length.
Here, the domain, as a subset of $C([0,1],\X)$, is endowed with the uniform topology and the target space with the topology induced by the flat norm. 

Indeed, given $0 < \alpha <1$, consider the Rickman rug, i.e.\ the metric space $(\R^2,\sfd_{\alpha})$, where
\begin{equation*}
    \sfd_\alpha((x,y),(x',y')):=\max\{|x-x'|^{\alpha},|y-y'|\},    
\end{equation*}
where $|\cdot|$ is the Euclidean distance in $\R$. In such a case, two vertical lines can be arbitrarily close, but the flat distance is bounded away from zero. More specifically, let $s >0$ and $T_s := \curr{\{0\} \times [0,1]}- \curr{\{s\}\times [0,1]} $.

Let $S \in \cN_2(\X)$, $R \in \cN_1(\X)$ such that $T_s=\partial S + R$. In particular, $\partial R = \partial T_s = \delta_{(0,1)}-\delta_{(0,0)} + \delta_{(s,0)}-\delta_{(s,1)}=:m$.
We claim that $\|\partial T_s\|_{\AEm(\X,\sfd_\ell)}=2$.
Given $f \in \Lip_1(\X,\sfd_\ell)$, we have that $\max\{|f(0,0)-f(0,1)|,|f(1,0)-f(1,1)|\}\le 1$, thus
\begin{equation*}
    \int f \,\d m \le 2,
\end{equation*}
proving $\|\partial T_s\|_{\AEm(\X,\sfd_\ell)} \le 2$. By choosing $f \in \Lip_1(\X,\sfd_\ell)$ defined as $f(x,y):=y\chi_{\{x\le s/2 \}}-y \chi_{\{x > s/2 \}}$ we get the converse inequality.
The claim, together with Lemma \ref{lemma:isomorphism_and_quasiconvexity} gives $\Mass(R_s) \ge 2$.
This implies $\Flatnorm(T_s) \ge 2$, and since $\Flatnorm(T_s)\le \Mass(T_s) \le 2$, we have $\Flatnorm(T_s)=2$ for every $s \in [0,1]$.
\end{remark}

\begin{lemma}[Measurability of fillings]\label{lemma:measurability_of_fillings}
    Let $(\X,\sfd)$ be a complete and separable metric space admitting
    a conical geodesic bicombing.
    Let $\eta$ be a finite Borel measure on $\Lip([0,1];\X)$, concentrated
    on a Borel set $A\subset\{\gamma\in\Lip([0,1];\X)\colon\ell(\gamma)\leq L\}$.
    Then
    \begin{equation*}
        T(f,\pi):=\int\curr{\gamma}(f,\pi)\,\d\eta(\gamma),\qquad (f,\pi)\in D^1(\X)
    \end{equation*}
    defines a metric $1$-current, which additionally satisfies
    \begin{equation*}
        \Flatnorm(T-\eta(A)\curr{\gamma_0})\leq 2(L+1)\eta(A) \mathrm{diam}(A,\sfd_\infty),
    \end{equation*}
    for every $\gamma_0\in A$.
\end{lemma}

\begin{proof}
To see that $T$ is well-defined, it is enough to show that the assumptions of \cref{lemma:integral_of_currents_are_currents} are met,
with measure space $(A,\mathcal{B}(A),\eta)$ and map $\gamma\in A\mapsto\curr{\gamma}$.
For $(f,\pi)\in D^1(\X)$, the map $\gamma\in A\mapsto \curr{\gamma}(f,\pi)$ is continuous, hence Borel,
and
\begin{equation*}
    \int\Mass(\curr{\gamma})\,\d\eta(\gamma)\leq L\eta(A)<\infty,
\end{equation*}
so $T$ is a metric $1$-current. \par
Assume $\mathrm{diam}(A,\sfd_\infty)<\infty$, for otherwise the statement is trivial.
For each $\gamma\in A$, let $S_\gamma\in\cM_1(\X)$, $R_\gamma\in\cM_1(\X)$
be given by \cref{lemma:homotopy_lemma} with $\gamma^1\equiv\gamma$ and $\gamma^0\equiv\gamma^0$.
They satisfy $\partial S_\gamma+R_\gamma = \curr{\gamma^0}-\curr{\gamma}$,
\begin{equation*}
    \Mass(S_\gamma)\leq 2L\mathrm{diam}(A,\sfd_\infty), \qquad
    \Mass(R_\gamma)\leq 2\mathrm{diam}(A,\sfd_\infty),
\end{equation*}
and moreover the maps $\gamma\in A\mapsto S_\gamma(f,\pi_1,\pi_2)$, $\gamma\in A\mapsto R_\gamma(f,\pi)$
are measurable for $(f,\pi_1,\pi_2)\in D^2(\X)$ and $(f,\pi)\in D^1(\X)$.
We can therefore apply \cref{lemma:integral_of_currents_are_currents} with measure space
$(A,\mathcal{B}(A),\eta)$ to see that
\begin{align*}
S(f,\pi_1,\pi_2)&:=\int_A S_\gamma(f,\pi_1,\pi_2)\,\d\eta(\gamma),\\
R(f,\pi)&:=\int_AR_\gamma(f,\pi)\,\d\eta(\gamma),
\end{align*}
define a metric $2$-current and $1$-current, respectively.
Moreover, they satisfy
$\partial S + R =\eta(A)\curr{\gamma^0}-T$
and
\begin{equation*}
    \Mass(S)+\Mass(R)\leq 2(L+1)\eta(A)\mathrm{diam}(A,\sfd_\infty).
\end{equation*}
This concludes the proof.
\end{proof}

\begin{theorem}[Geodesic approximation in conical geodesic bicombing metric space]
\label{thm:geodesic_approximation}
Let $(\X,\sfd)$ be a complete and separable metric space admitting a conical geodesic bicombing.
Let $N \in \cN_1(\X)$. Then for every $\varepsilon>0$, there exists an $n \in \mathbb{N}$ and $P=\sum_{i=1}^n a_i\curr{\gamma_i}$, where $\gamma_i \in {\rm Geo}(\X)$ and $a_i\in\R$ for every $1 \le i \le n$, such that $\Flatnorm(N-P) \le \varepsilon$ and $\Mass(P)\leq\Mass(N)$.
Moreover, if $\X$ is a closed convex subset of a Banach space, we can choose $\gamma_i$ to be line segments. 
\end{theorem}

\begin{proof}
    \textsc{Step 1} (Approximation by a finite sum of curves with multiplicities).
    We first show that for $\varepsilon>0$ there are Lipschitz curves $\gamma_i\colon[0,1]\to\X$, $a_i\in\R$,
    such that
    \begin{equation*}
        \Flatnorm \left(N-\sum_{i=1}^na_i\curr{\gamma_i}\right)\leq \varepsilon,
        \qquad \sum_{i=1}^n|a_i|\ell(\gamma_i)\leq\Mass(N).
    \end{equation*}
    Since $N \in \cN_1(\X)$, by a result of Paolini-Stepanov \cite[Theorem 3.1]{PaoliniStepanov_cycles}, there exists a finite Borel measure $\eta$ concentrated on 
    rectifiable curves such that
    \begin{equation}
    \label{eq:PS_property}
        N =\int \curr{\gamma}\,\d \eta(\gamma)\quad\text{ and }\quad\Mass(N)=\int \ell(\gamma)\,\d \eta(\gamma).
    \end{equation}
    Let $\mathcal{C}_L:=\{\gamma\in C([0,1];\X)\colon\ell(\gamma)\leq L\}$ and set
    \begin{equation*}
        N_L:=\int_{\mathcal{C}_L}\curr{\gamma}\,\d\eta(\gamma).
    \end{equation*}
    Since $\eta$ is concentrated on rectifiable curves, $\Flatnorm(N-N_L)\leq\Mass(N-N_L)\leq\int_{\mathcal{C}_L^c}\ell(\gamma)\,\d\eta(\gamma)\rightarrow 0$ as $L\rightarrow \infty$.
    We also have $\Mass(N_L)\leq\Mass(N)$. We can therefore assume
    $N=N_L$ for some $L>0$, i.e.\ that $\eta$ is concentrated on
    $\mathcal{C}_L$.\par
    Fix $\varepsilon>0$.    
    Since $(\X,\sfd)$ is separable, so is $C([0,1],\X)$, and there is a countable dense set $\{\tilde{\gamma}_i\}_i$.
    Set $A_i :=\mathcal{C}_L\cap ( B_\varepsilon(\tilde{\gamma}_i)\setminus \cup_{j<i}B_\varepsilon(\tilde{\gamma}_j))$, so that $\{A_i\}_i$ is a pairwise disjoint cover of $\mathcal{C}_L$.
    Set $N^n:=\int_{\cup_{i=1}^nA_i}\curr{\gamma}\,\d\eta(\gamma)$ and note that, similarly as before,
    $\Mass(N^n)\leq\Mass(N)$ and
    $\Mass{(N-N^n)}\rightarrow0$ as $n\rightarrow\infty$.
    We now approximate $N^n$.
    For $1\leq i\leq n$, pick $\gamma_i \in A_i$ such that
    \begin{equation}
        \label{eq:choice_of_gamma}
        \eta(A_i)\ell(\gamma_i) \le \int_{A_i} \ell(\gamma)\,\d \eta(\gamma),
    \end{equation}
     and define $T:=\sum_{i=1}^n \eta(A_i) \curr{ {\gamma_i} }$.
    By \eqref{eq:choice_of_gamma} and \eqref{eq:PS_property}, we have $\Mass(T)\le \Mass(N)$ and,
    by Lemma \ref{lemma:homotopy_lemma},
     \begin{equation*}
         \Flatnorm(N^n-T)\leq\sum_{i=1}^n2(L+1)\eta(A_i)2\varepsilon=4(L+1)\eta(C([0,1],\X))\varepsilon.
     \end{equation*}
     Since $\varepsilon>0$ was arbitrary and $\Mass(N-N^n)\rightarrow 0$, this concludes the first step.
     
    \textsc{Step 2} (Approximation by a finite sum of geodesics)
    By triangle inequality (for the flat-norm), it is enough to show that we can approximate
    each $\curr{\gamma_i}$ by a finite sum of geodesics with total length less than $\ell(\gamma_i)$.\par
    Fix $i$ and set $\gamma:=\gamma_i$.
    For $\varepsilon>0$, let $\mathcal{T}^\varepsilon=\{0=t_0^\varepsilon<\cdots<t_{m(\varepsilon)}^\varepsilon=1\}$
    be a partition of $[0,1]$ with mesh size $\max_{1\leq j\leq m(\varepsilon)}(t_j-t_{j-1})<\varepsilon$.
    We further assume $\mathcal{T}^{\varepsilon_2}\subset\mathcal{T}^{\varepsilon_1}$ for $0<\varepsilon_1<\varepsilon_2$.
    For $1\leq j\leq m(\varepsilon)$ and $t\in[t_{j-1}^\varepsilon,t_j^\varepsilon]$, set
    \begin{equation*}
    \sigma^\varepsilon(t):=\sigma(\gamma_{t_{j-1}^\varepsilon},\gamma_{t_j^\varepsilon},(t-t_{j-1}^\varepsilon)/(t_{j}^\varepsilon-t_{j-1}^\varepsilon)),
    \end{equation*}
    where $\sigma$ is a conical geodesic bicombing on $\X$.
    Then $\sigma^\varepsilon:[0,1]\to\X$ is a continuous piecewise geodesic curve satisfying
    \begin{equation}
    \label{eq:length_approximation}
       \Mass(\curr{\sigma^\varepsilon})\leq \ell(\sigma^\varepsilon)=\sum_{j=1}^{m(\varepsilon)} \sfd(\gamma_{t_{j-1}^\varepsilon},\gamma_{t_{j}^\varepsilon}) \le \ell(\gamma).
    \end{equation}
    The curves $\sigma^\varepsilon$ converge uniformly to $\gamma$ and have equibounded lengths, so \cref{lemma:homotopy_lemma}
    shows that $\Flatnorm_X(\curr{\sigma^\varepsilon}-\curr{\gamma})\rightarrow 0$ as $\varepsilon\rightarrow 0$.
    Since $\curr{\sigma^\varepsilon}$ can be written as a finite sum of geodesics,
    \begin{equation*}
        \curr{\sigma^\varepsilon}=\sum_{j=1}^{m(\varepsilon)}\curr{\sigma^\varepsilon\restr{[t_{j-1}^\varepsilon,t_j^\varepsilon]}},
    \end{equation*}
    this concludes the proof.
\end{proof}

\subsection{Homotopy formula for general \texorpdfstring{$k$}{k} in Banach spaces}

The goal of this section is to prove an homotopy formula for $k$-currents in a Banach space $(\B,\|\cdot\|)$. We introduce some notation. Let $\phi, \psi\in \Lip(\B;\B)$.
Let $h\colon [0,1]\times \B \to \B$ be the affine homotopy between Lipschitz functions defined as $h(t,x):=t\phi(x)+(1-t)\psi(x)$. For brevity, for every $f \colon \B \to \R$, we denote $f_t:=f \circ h(t,\cdot)$ and $\hat \pi_i=(\pi_1,...\pi_{i-1},\pi_{i+1},...\pi_k)$.
Note that for $f\in \Lip(\B)$ we have that for every $x \in \B$ 
\begin{equation} \label{partialhom}
	\left|\frac{\partial (f\circ h)}{\partial t}(t,x)\right|\leq \Lip f(x)\, |\phi(x)-\psi(x)|
\end{equation}
at every point $t$ of differentiability of $f\circ h(\cdot,x)$ and
\begin{equation} \label{liphom}
	\Lip f_t(x)\leq \Lip f(x)\,\max\{\Lip \psi(x),\Lip\phi(x)\}
\end{equation}
for every $x \in \B$ and $t \in [0,1]$.
\par

For $T\in\kcur (\B)$ and $h$ as before, we define the $(k+1)$-metric functional $H_h(T)$ as
\begin{equation*}
    H_h(T)(f,\pi):=\sum_{i=1}^{k+1}(-1)^{i+1}\int_{0}^{1}T\left(f\circ h\frac{\partial(\pi_i\circ h)}{\partial t}, \hat\pi_i\circ h\right)\,\d t.    
\end{equation*}
For simplicity of notation, when there is no ambiguity we denote it by $H(T)$.
A similar construction was used by \cite[Section 10]{AK00} to construct a cone over a metric $k$-current.

\begin{proposition}[Properties of $H(T)$]\label{lemma:homotopyform}
	Let $T\in \kncur(\B)$ have bounded support. Then
	\begin{enumerate}
		\item $\mass(H(T))\leq (k+1) \int \|\phi-\psi\| \max\{\Lip \psi(x),\Lip \phi(x)\}^k \,\d\|T\|$; \label{homotopy:mass}
		\item $H(T)\in \ncur_{k+1}(\mathbb{B})$; \label{homotopy:current}
		\item $\phi_\# T-\psi_\# T=\partial H(T)+H(\partial T)$.\label{homotopy:boundary}
	\end{enumerate} 
\end{proposition}

\begin{proof}
	We only need to show (\ref{homotopy:mass}). Once (\ref{homotopy:mass}) is shown, (\ref{homotopy:current}) and (\ref{homotopy:boundary}) are obtained by following the proof of \cite[Proposition 10.2]{AK00} exactly. Let $(f,\pi)\in D^{k+1}(\B)$ with $\LIP(\pi_i)\leq 1$ for $1\leq i\leq k+1$.
	We have
	\begin{align*}
		|H(T)&(f, \pi)|\leq  \sum_{i=1}^{k+1}\int_{0}^{1}\left|T\left(f\circ h\,\frac{\partial(\pi_i\circ h)}{\partial t}, \hat\pi_i\circ h\right)\right|\,\d t\\
            & \leq \sum_{i=1}^{k+1}\int_{0}^{1}\int \left|f\circ h\,\frac{\partial(\pi_i\circ h)}{\partial t}\right|(x) \prod_{j\neq i} \Lip(\pi_{j}\circ h)(x)\, \d \|T\|(x)\,\d t \\
		&\stackrel{\eqref{liphom}}{\leq} \sum_{i=1}^{k+1}\int_{0}^{1}\int \left| f\circ h\,\frac{\partial(\pi_i\circ h)}{\partial t}\right|(x)\prod_{j \neq i} \Lip \pi_j (x) (\max\{\Lip\psi(x),\Lip\phi(x)\})^k\,\d\|T\|(x)\,\d t\\
		&\stackrel{\eqref{partialhom}}{\leq} (k+1)  \int_0^1 \int_\mathbb{B} (|f|\circ h)(t,x)\,\|\phi(x)-\psi(x)\| (\max\{\Lip\psi(x),\Lip\phi(x)\})^k \,\d\|T\|(x)\,\d t \\
		&\leq (k+1)\int_\mathbb{B} |f|\,\d\nu,
    \end{align*}
where the measure $\nu$ is defined as
    \[\nu(A)=\int_0^1 \int_{h_t^{-1}(A)}\|\phi(x)-\psi(x)\| (\max\{\Lip \psi(x),\Lip  \phi(x)\})^k \,\d\|T\|(x)\,\d t\]
for all Borel sets $A \subset \B$.

This proves $\|H(T)\| \le (k+1)\nu$ as measures,
and in particular
\begin{equation*}
 \Mass(H(T))\le (k+1)\nu(\B)\le (k+1) \int \|\phi(x)-\psi(x)\| (\max\{\Lip \psi(x),\Lip \phi(x)\})^k \d\|T\|(x).
\end{equation*}
\end{proof}

\section{Structure of flat chains in Banach spaces}
\label{sec:alberti_marchese}

The aim of this section is to extend the results of \cite[Section 3]{AlbertiMarcheseFlatChains} to the case of an ambient Banach space instead of Euclidean space for the case of $1$-dimensional metric currents.
The goal is essentially to show that any $1$-current can be obtained by restricting a normal $1$-current to a Borel set.
Building upon the work carried out in the previous sections, we can essentially follow the argument of their paper. There is a minor modification: in a special case, we can even prove that a metric $1$-current is the restriction of a normal $1$-current to a \emph{closed} set, as opposed to Borel. This will be important for the application in Section \ref{sec:representation_results}
as it will allow us to consider restrictions of fragments to closed sets,
see \eqref{eq:application_albmarchese}. The difficulty is that a fragment, by definition, needs to have a compact domain, thus a fragment restricted to a Borel set need not be a fragment itself.

Throughout this section, $(\B,\|\cdot\|)$ is a Banach space. Given a set $A \subset \B$, we denote by ${\rm int} (A)$ its topological interior. Given $A\subset B \subset \B$, we denote by ${\rm int}(A , B)$ the topological interior of $A$ in $B$ equipped with the topology inherited from $\B$.

We need to fix some notation concerning polyhedral chains and the flat norm. Let $0 < k < \dim \B$ and let $\sigma$ be a $k$-simplex. In particular $\sigma$ is the convex hull of $(k+1)$-points $a_i$ for $i=0,\dots,k$. Let $S_\sigma:= {\rm span}\{a_i-a_0:\, i=1,\dots,k\}=:S_\sigma$. Given a $k$-simplex $\sigma$, $v_i \in S_\sigma$ for $i=1,\dots,k$, we define the metric $k$-current $\curr{\sigma, v_1,\dots,v_k}$ as
\begin{equation*}
    \curr{\sigma, v_1,\dots,v_k}(f,\pi_1,\dots,\pi_k):=\int_\sigma f \la \d \pi_1 \wedge \dots \wedge \d \pi_k, v_1\wedge \dots \wedge v_k \ra\,\d \mathcal{H}^k.
\end{equation*}
Here, by $\d \pi_i$ we denote the differential of the Lipschitz function $\pi_i$ restricted to $S_\sigma + a$, where $a$ is a vertex of $\sigma$. This is well-defined at $\mathcal{H}^k$ a.e.\ point of $\sigma$ because of Rademacher's theorem.

We define the set of \emph{polyhedral $k$-chains} supported on a set $C \subset \mathbb{B}$, denoted by $\cP_k(C)$, as the set of finite sums 
\begin{equation*}
    P=\sum_{i=1}^N a_i\, \curr{\sigma_i,v_1^i,\dots,v_k^i}
\end{equation*}
where $\sigma_i$ is a $k$-simplex, $v^i_j\in S_{\sigma_i}$ for $i=1,\dots,N$, $j=1,\dots,k$ and $a_i \in \R$ and the intersection $\sigma_{i_{0}}\cap \sigma_{i_1}$ is a subset of the boundary of $\sigma_{i_0}$ whenever $i_0\not= i_1$.

If we highlight the dependence on the norm in the definition of Hausdorff measure, we write $\mathcal{H}_{\|\cdot\|}^k$.
We define the space of flat $k$-chains and we denote by $(\cF_{k}(C),\Flatnorm(\cdot))$ as the completion of $\cP_{k}(C)$ with respect to $\Flatnorm(\cdot)$. \par

We recall that elements of $\cF_k(C)$ can be regarded as metric $k$-currents.
This follows by the fact that the flat and mass closures of normal $k$-currents coincide together with the fact that polyhedral $k$-chains are normal $k$-currents. We define the \emph{relaxed mass} norm on the space $\cF_k(C)$ as 
\begin{equation*}
    \relMass(T):=\inf \left\{\liminf_{n \to \infty}\Mass(P_n)\,:\,P_n \in \cP_k(C),\,\lim_{n\to \infty}\Flatnorm(P_n-T) = 0  \right\} 
\end{equation*}
for every $T \in \cF_k(C)$. The main result follows. 

\begin{theorem}
    \label{theorem:main_alberti_marchese}
    Let $C \subset \B$ be closed, convex and bounded with $0\in{\rm int}(C)$. Assume $\dim \B >1$.
    Let $T \in \cM_{1}(C)$ and suppose $\|T\|$ is inner regular by compact sets.
    Then, for any $\varepsilon>0$, there exists $N \in \cN_{1}(C)$ and a Borel set $B\subset \B$ such that
    \begin{itemize}
        \item[(i)] $\partial N=0$,
        \item[(ii)] $\Mass(N)\leq (2+\varepsilon)\Mass(T)$,
        \item[(iii)] $N\restr{B} = T$.
    \end{itemize}

    Moreover, if the support of $T$ is contained in an affine hyperplane, then we can take $B={\rm supp}(T)$.
\end{theorem}

Recall that, by definition, an affine hyperplane in a Banach space is a level set of a non-zero linear continuous functional.

Following \cite{AlbertiMarcheseFlatChains}, we prove the following preliminary results, which we prove in slightly higher generality than what
is strictly necessary for our purposes.

\begin{lemma}
\label{lemma:rescalement_in_the_interior}
Let $C \subset \B$ be closed, convex and bounded and such that $0\in{\rm int}(C)$.
    Let $T \in \cF_k(C)$ for some $0<k< \dim \B$, with $k \in \N$. Then, for every $\varepsilon>0$, there exists $P \in \cP_{k}(C)$ such that $\Flatnorm(T-P) \le \varepsilon$,
    ${\rm supp} (P) \subset {\rm int}(C)$ and $\Mass(P) \le (1+\varepsilon) \relMass(T)$.
\end{lemma}

\begin{proof}
    For every $\varepsilon >0$, there exists a polyhedral chain $P' \in \Poly{k}(C)$ such that $\Flatnorm(T-P') \le \frac{\varepsilon}{2}$ and such that $\Mass(P') \le (1+\frac{\varepsilon}{2}) \relMass(T)$. Let $h(t,x):=t x$ and denote $h_t(\cdot)=h(t,\cdot)$. Define $P:= {h_{1-\eta}}_*P'$ for $\eta$ to be chosen later and we claim that $P$ satisfies the conclusion of the theorem. Notice that
    $P \in \cP_{k}(C)$.
    We apply Proposition \ref{lemma:homotopyform} with $\psi:={\rm Id}$ and $\phi:=h_{1-\eta}$, thus having 
    \begin{equation*}
        \Flatnorm(P-{h_{1-\eta}}_*P) \le \eta ((k+1)\Mass(P)+k \Mass(\partial P)).
    \end{equation*}

    Moreover, we claim that ${\rm supp}(P) \subset {\rm int}(C)$. Indeed, since $0 \in {\rm int}(C)$, we have $B(0,\delta) \subset C$ for some $\delta$. Since $C$ is convex, we have that $(1-t)B(0,\delta) +t {\rm supp} (P') \subset C$ for every $t \in [0,1]$, implying that $t {\rm supp} (P') \in {\rm int}(C)$. We conclude noticing that ${\rm supp}(P)={\rm supp}({h_{1-\eta}}_* P') = (1-\eta) {\rm supp}(P')$.
    
    Moreover, since $h_t \colon \B \to \B$ is $t$-Lipschitz for every $t \in [0,1]$, we have that 
    \begin{equation*}
        \Mass{({h_t}_* P')} \le t^k \Mass{(P')}.
    \end{equation*}
    Thus, the proof follows by choosing $\eta$ sufficiently small.
\end{proof}

\begin{lemma}
    \label{lemma:constructu_current_orthogonal_to_mu}
   Let $C \subset \B$ be closed, convex and bounded and such that $0\in{\rm int}(C)$.
   Let $T \in \cF_{k}(C)$ for some $0<k< \dim \B$, with $k \in \N$ and let $\mu$ be a Radon measure. Let $H \subset \B$ be an affine hyperplane. For every $\varepsilon >0$ there exists $P \in \cP_{k}(C)$ such that $\mu \perp \|P\|$, $\Flatnorm(P-T) \le \varepsilon$, $\Mass(P) \le (1+\varepsilon) \relMass(T)$ and $\|P\|(H)=0$.
\end{lemma}

\begin{proof}
    We can assume without loss of generality that $P \neq 0$. 
    We apply Lemma \ref{lemma:rescalement_in_the_interior} and we have there exists $P' \in \cP_{k}(C)$ such that $\Flatnorm(T-P') \le \frac{\varepsilon}{2}$,
    ${\rm supp} (P') \subset {\rm int}(C)$ and $\Mass(P') \le (1+\varepsilon) \relMass(T)$.
    
    Since ${\rm supp}(P')$ is compact and ${\rm int}(C)$ open, there is $t_0>0$
    such that for $v\in\B$ with $\|v\|\leq 1$, we have
    \begin{equation*}
        {\rm supp}(P')+t v \subset {\rm int}(C),\qquad \text{for every }t \in [0,t_0].
    \end{equation*}

    We write $P'=\sum_{i=1}^N a_i\, \curr{\sigma_i,v_1^i,\dots,v_k^i}$ for some $N \in \N$ with $0<a_i \in \R$ and define the set
    \begin{equation*}
        \mathcal{D}_0:=\left\{ \lambda(x-y):\, x,y \in \sigma_i \text{ for some }i,\, \lambda \in\R
        \right\}.
    \end{equation*}
    Let $p\in H$. Since $H$ is a hyperplane and $\mathcal{D}_0$ is a finite union of $k$-dimensional subspaces
    of $\B$ and $k<\dim\B$, there is $w\in\B$, $\|w\|=1$, with $w\notin\mathcal{D}_0$
    and $w+p\notin H$.
    We define $\tau_t \colon \B \to \B$ as the map $\tau_t(x):=x+tw$, $t \in [0,t_0]$. We define $P_t:={\tau_t}_* P'$. 
    Let $t_1 \in (0,t_0)$ to be chosen later. Given $t\in [0,t_1]$, we apply Proposition \ref{lemma:homotopyform} with the choices $\psi:={\rm Id}$ and $\phi:=\tau_{t}$ and we have
    \begin{equation*}
        \Flatnorm(P'-P_t)\le t_1((k+1)\Mass(P)+k \Mass(\partial P)). 
    \end{equation*}
    We choose $t_1:= \varepsilon (2((k+1)\Mass(P)+k \Mass(\partial P)))^{-1}$.
    
    Arguing as in \cite[Corollary 3.2]{AlbertiMarcheseFlatChains}, we have that $\mu \perp \|P_t\|$ for all but countably many $t\in [0,t_1]$.
    Since the measures $\{\|P_t\|\colon t\in [0,t_1]\}$ are mutually singular,
    we then find one $t_2\in [0,t_1]$ such that $\|P_{t_2}\|\perp \mu$ and $\|P_{t_2}\|(H)=0$.
    Set $P:=P_{t_2}$ and observe that
    the polyhedral current $P$ satisfies the conclusion of the proposition. Indeed, we have that $\Flatnorm(P-P') \le \varepsilon$ and, since $\tau_{t_2}$ is an isometry,
    \begin{equation*}
        \Mass(P)=\Mass(P') \le (1+\varepsilon) \relMass(T).
    \end{equation*}
\end{proof}

\begin{lemma}
\label{lemma:relmass_equal_to_mass}
Let $(\mathbb{B},\|\cdot\|)$ be a Banach space and let $C$ be as before. 
We have that $\relMass(T) \ge \Mass(T)$ for every $T \in \cM_k(C)$.
Moreover, if $T\in \cM_1(C)$ and $\|T\|$ is inner regular by compact sets, we have that $T\in \cF_1(C)$ and $\relMass(T)=\Mass(T)$ for $T\in \cM_1(C)$.
\end{lemma}

\begin{proof}
By lower semicontinuity of $\Mass$ with respect to flat convergence, we have that $\relMass(T) \ge \Mass(T)$ for every $T \in \cM_1(C)$.

For the converse inequality, we fix $T\in \cM_1(C)$ and note that, arguing as in Corollary \ref{coro:approximation_by_normal_currents}, we may assume $C$ to be separable. We notice that $(C,\|\cdot\|)$, as a convex subset of a Banach space, is a metric space that admits a conical geodesic bicombing, so we are in position to apply Theorem \ref{thm:geodesic_approximation}.  Moreover, since $(C,\|\cdot\|)$ is quasiconvex, we apply Corollary \ref{coro:approximation_by_normal_currents} and we can approximate $T$ with elements in $\cN_1(C)$. Both results together and a diagonal argument give that for every $T \in \cM_1(C)$ there exists a sequence of polyhedral chains $\{P_n\}_n \subset \cP_1(C)$ such that $\Flatnorm(P_n-T)\to 0$ and 
$\Mass(P_n) \to \Mass(T)$. This proves that $T\in \cF_1(C)$ and that $\Mass(T)\ge \relMass(T)$.
\end{proof}

\begin{proposition}
\label{prop:construction_orthogonal_rectifiable_current_sameboundary}
 Let $(\mathbb{B},\|\cdot\|)$ be a Banach space. Let $C \subset \B$ be closed, convex and bounded with $0\in{\rm int}(C)$, $\mu$ a nonnegative finite Borel measure and $H$ an affine hyperplane in $\B$. Let $T \in \cM_1(C)$ such that $\|T\|$ is inner regular by compact sets.
 Then, for $\varepsilon>0$, there is a rectifiable $1$-current $R\in \cM_1(C)$ such that 
 \begin{itemize}
     \item[(i)] $\partial R = \partial T$;
     \item[(ii)] $\mu \perp \|R\|$ and $\|R\|(H)=0$;
     \item[(iii)] $\Mass(R) \le (1+\varepsilon)\Mass(T)$.
 \end{itemize}
\end{proposition}

\begin{proof}
Choose $\delta$ such that $(1+\delta)^2 \le 1+\varepsilon$, set $C_0:=1+\delta$, and define $\varepsilon_n:=\Mass(T) \delta 2^{-n}$.
We apply Lemma \ref{lemma:constructu_current_orthogonal_to_mu} to $T$ and we get $P_0 \in \cP_{1}(C)$ such that 
\begin{equation*}
    T-P_0=R_0+\partial S_0,\quad \Mass(R_0)+ \Mass(S_0) \le \varepsilon_0,\quad \Mass(P_0)\le C_0 \relMass(T)=C_0\Mass(T).
\end{equation*}
for some $R_0 \in \cM_{1}(C)$, $S_0 \in \cN_{2}(C)$. Notice that the fillings for computing the flat norms in Lemma \ref{lemma:rescalement_in_the_interior} and Lemma \ref{lemma:constructu_current_orthogonal_to_mu}.
We have also used that $\relMass=\Mass$ on metric $1$-currents (with inner regular mass measure) by applying Lemma \ref{lemma:relmass_equal_to_mass}.

We have that $\partial T= \partial P_0 + \partial R_0$. We are in position to apply again Lemma \ref{lemma:constructu_current_orthogonal_to_mu} to $R_0$, because $R_0$ is a metric $1$-current, hence a flat $1$-chain by Lemma \ref{lemma:relmass_equal_to_mass} and $\relMass(R_0)=\Mass(R_0)$. We choose $\varepsilon=\varepsilon_1$ and we have that there exists $P_1\in \cP_{1}(C)$ such that 
\begin{equation*}
    R_0-P_1=R_1+\partial S_1,\quad \Mass(R_1)+ \Mass(S_1) \le \varepsilon_1,\quad \Mass(P_1)\le C_0 \Mass(R_0)
\end{equation*}
for some $R_1 \in \cM_{1}(C)$, $S_1 \in \cN_{2}(C)$. In particular $\partial T= \partial P_0 +\partial P_1 +\partial R_1$.
We repeat the construction iteratively, thus obtaining a remainder term $R_n$ to which we apply the previous lemma with the choice $\varepsilon=\varepsilon_n$.
This gives a sequence of polyhedral chains $P_n \in \cP_{1}(C)$ and remainder terms $R_n \in \cM_{1}(C)$ such that 
\begin{equation*}
    \partial T = \partial \left( \sum_{i=0}^n P_i \right) + \partial R_n,\qquad \Mass(P_n) \le C_0 \Mass(R_{n-1})\le C_0 \varepsilon_{n-1} ,\qquad \Mass(R_n) \le \varepsilon_n
\end{equation*}
for every $n \in \N$.
Moreover, we have that $\mu \perp \|P_i\|$ and $\|P_i\|(H)=0$ for every $0\leq i \le n$.
Since $\sum_{n=0}^\infty \Mass(P_n) < \infty$, we have that $\sum_{i=0}^n P_i$ is Cauchy with respect to $\Mass(\cdot)$, thus converges to some $1$-rectifiable current $R \in \cM_1(C)$.
For $g\colon \X\to\R$ Lipschitz, we have
\begin{equation}
\label{eq:pointwise_estimate_partial_sum}
    T(1, g)= \partial T(g)= \sum_{i=0}^n \partial  P_i(g)  + \partial R_n(g)= \sum_{i=0}^n P_i(1,g) + R_n(1,g).
\end{equation}
Since $|R_n(1,g)| \le \Lip(g) \Mass(R_n)$ for every $n \in \N$ and $\Mass(R_n) \to 0$, by taking the limit in \eqref{eq:pointwise_estimate_partial_sum} we have that $\partial R=\partial T$.

To prove the mass estimate, we compute
\begin{equation*}
    \Mass(R)\le \sum_{i=0}^\infty \Mass(P_i) \le C_0 \Mass(T) + C_0 \sum_{i=0}^\infty \varepsilon_i \le C_0 \Mass(T) + C_0\delta\Mass(T)\le (1+\varepsilon)\Mass(T).
\end{equation*}

It remains to check that $\mu \perp \|R\|$ and $\|R\|(H)=0$. Since $R=\sum_{i=0}^\infty P_i$ with convergence in mass,
it can be readily checked, using $\|\sum_{i=0}^nP_i\|\leq\sum_{i=0}^n\|P_i\|$, that $\|R\|\leq \sum_{i=0}^\infty \|P_i\|$.
Hence $\|R\| \perp \mu$ and $\|R\|(H)=0$ follow from $\|P_i\| \perp \mu$ and $\|P_i\|(H)=0$ for each $i \in \N$.
\end{proof}

\begin{proof}[Proof of Theorem \ref{theorem:main_alberti_marchese}]
We apply Proposition \ref{prop:construction_orthogonal_rectifiable_current_sameboundary} with $\mu =\|T\|$ and take $N:=T-R$. Moreover, since $\|T\|\perp \|R\|$, there exists a Borel set $B \subset \B$ such that $\|T\|$ is concentrated on $B$ and $\|R\|(B)=0$. Lastly, if ${\rm supp}(T)$ is contained in an affine hyperplane $H$, then $N\restr{H}=N\restr{{\rm supp}(T)}=T$.
\end{proof}

\section{Representation result for metric $1$-currents in Banach spaces}
\label{sec:representation_results}
We begin with some notation and measurability lemmas, cf.\ \cite[Lemma A.1]{EB_Teri_p_weak}.
Set
\begin{equation*}
\overline{\Gamma}(\X):=\{(t,\gamma)\colon t\in {\rm dom}(\gamma),\gamma\in\Gamma(\X)\},
\end{equation*}
endow it with the subspace topology $\overline{\Gamma}(\X)\subset [0,1]\times\Gamma(\X)$,
and let $e\colon\overline{\Gamma}(\X)\to\X$ be the continuous map $e(t,\gamma):=\gamma_t$.
\begin{lemma}\label{lemma:measurability_enhanced_frag}
Let $(\X,\sfd)$ be a separable metric space and $f\colon \X\to\R$ a continuous function.
Then the sets
\begin{equation*}
\begin{aligned}
    \mathit{MD}&:=\{(t,\gamma)\in\overline{\Gamma}(\X)\colon t\text{ is a limit point of }{\rm dom}(\gamma)\text{ and }|\dot{\gamma}_t|\text{ exists}\} \\
    D(f)&:=\{(t,\gamma)\in\overline{\Gamma}(\X)\colon t\text{ is a limit point of }{\rm dom}(\gamma)\text{ and }(f\circ\gamma)_t'\text{ exists}\}
\end{aligned}
\end{equation*}
are Borel and the maps
\begin{equation*}
\begin{aligned}
(t,\gamma)\in\overline{\Gamma}(\X)&\mapsto
\left\{
\begin{array}{cc}
|\dot{\gamma}_t|, & (t,\gamma)\in\mathit{MD}\\
0, &\text{otherwise}
\end{array}
\right.
\\
(t,\gamma)\in\overline{\Gamma}(\X)&\mapsto
\left\{
\begin{array}{cc}
(f\circ\gamma)_t', & (t,\gamma)\in D(f)\\
0, &\text{otherwise}
\end{array}
\right.
\end{aligned}
\end{equation*}
are Borel measurable.
\end{lemma}
\begin{proof}
We only prove the part of the statement regarding $D(f)$ and $(t,\gamma)\mapsto (f\circ\gamma)_t'$;
the other one is similar.
For $r>0$, the maps $(t,\gamma)\mapsto \inf(B(t,r)\cap{\rm dom}(\gamma))$,
$(t,\gamma)\mapsto\sup(B(t,r)\cap{\rm dom}(\gamma))$ are upper and
lower-semicontinuous, respectively.
It is then not difficult to construct, for each $r>0$, a Borel map $\tau_r\colon\overline{\Gamma}(\X)\to [0,1]$,
with $\tau_r(t,\gamma)\in \overline{B}(t,r)\cap{\rm dom}(\gamma)$,
such that $\tau_r(t,\gamma)=t$
if and only if $B(t,r)\cap{\rm dom}(\gamma)=\{t\}$.
In particular, since $\tau(t,\gamma):=t$ is continuous,
\begin{equation*}
L:=\{(t,\gamma)\in\overline{\Gamma}(\X)\colon t\text{ is a limit point of }{\rm dom}(\gamma)\}
=\bigcap_{\substack{r\in\Q \\ r>0}}\{\tau_r\neq\tau\}
\end{equation*}
is a Borel set.
Observe that, for $(t,\gamma)\in  L$, it holds $(t,\gamma)\in D(f)$
if and only if for every $\varepsilon>0$ there are $R>0$ and $D\in\R$
such that for each $0<r<R$ we have
\begin{equation*}
    |(f\circ\gamma)_s-(f\circ\gamma)_t-D(s-t)|\leq\varepsilon r
\end{equation*}
for all $s\in B(t,r)\cap{\rm dom}(\gamma)$.
Since the function $(x_1,x_2,t_1,t_2)\in \X^2\times[0,1]^2\mapsto |f(x_2)-f(x_1)-D(t_2-t_1)|$
is continuous,
one can see that the map $\tilde{\Phi}_{D,r}\colon\overline{\Gamma}(\X)\to[0,\infty)$
\begin{equation*}
    \tilde{\Phi}_{D,r}(t,\gamma):=\sup_{s\in B(t,r)\cap{\rm dom}(\gamma)}
    \frac{|(f\circ\gamma)_s-(f\circ\gamma)_t-D(s-t)|}{r}
\end{equation*}
is lower-semicontinuous and therefore so is
$\Phi_{D,R}:=\sup_{0<r<R}\tilde{\Phi}_{D,r}$.
Recalling the previous discussion, we have
\begin{equation*}
D(f)=L\cap\bigcap_{\substack{\varepsilon\in\Q \\ \varepsilon>0}}
    \bigcup_{\substack{R,D\in\Q \\ R>0}}\{\Phi_{D,R}\leq\varepsilon\},
\end{equation*}
showing that $D(f)$ is Borel.
It remains to prove measurability of the map $F\colon\overline{\Gamma}(\X)\to\R$
given by
$F(t,\gamma):=(f\circ\gamma)_t'$ for $(t,\gamma)\in D(f)$, $F(t,\gamma):=0$ otherwise.
Since $\X$ is separable, so is $\Gamma(\X)$,
and the Borel $\sigma$-algebra of $[0,1]\times\Gamma(\X)$
is therefore generated by open sets of the form $A\times B$ with $A\subset [0,1]$,
$B\subset\Gamma(\X)$.
The maps $(t,\gamma)\in\overline{\Gamma}(\X)\mapsto (\tau_r(t,\gamma),\gamma)$,
$(t,\gamma)\mapsto e(\tau_r(t,\gamma),\gamma)=\gamma_{\tau_r(t,\gamma)}$
are then easily seen to be Borel.
Let $F_r\colon\overline{\Gamma}(\X)\to\R$ be given by
\begin{equation*}
    F_r(t,\gamma):=\frac{(f\circ\gamma)_{\tau_r(t,\gamma)}-(f\circ\gamma)_t}{\tau_r(t,\gamma)-t}
\end{equation*}
for $(t,\gamma)\in D(f)$ and $F_r(t,\gamma):=0$ otherwise.
It is then clear that $F_r$ is Borel for every $r>0$ and, since $|\tau_r(t,\gamma)-t|\leq r$,
fixing $r_j\rightarrow 0$ we see that
$F=\lim_{j\rightarrow\infty}F_{r_j}$ is also Borel.
\end{proof}

\begin{lemma}\label{lemma:meas_frag_curr}
Let $(\X,\sfd)$ be a complete and separable metric space.
Then, for $f\in\mathcal{B}^\infty(\X)$, $\pi\in\Lip(\X)$, and $B\subset\X$ Borel, the
maps $\gamma\in\Gamma(\X)\mapsto\curr{\gamma}(f,\pi)$, 
$\gamma\mapsto \|\curr{\gamma}\|(B)$, and $\gamma\mapsto\int_{\gamma^{-1}(B)} |\dot{\gamma}_t|\,\d t$ are Borel measurable.
\end{lemma}
\begin{proof}
Define $\mathit{MD}$ as in \cref{lemma:measurability_enhanced_frag}, let
$B\subset\X$ be Borel,
and let $G\colon [0,1]\times\Gamma(\X)\to\R$ be given by $G(t,\gamma):=|\dot{\gamma}_t|$
for $(t,\gamma)\in\mathit{MD}\cap e^{-1}(B)$ and $G(t,\gamma):=0$ otherwise.
Since $\overline{\Gamma}(\X)$ is closed in $[0,1]\times\Gamma(\X)$ and
$e\colon\overline{\Gamma}(\X)\to\X$ continuous, $G$ is Borel by \cref{lemma:measurability_enhanced_frag}.
Then, $\gamma\in\Gamma(\X)\mapsto\int_0^1G(t,\gamma)\,\d t=\int_{\gamma^{-1}(B)}|\dot{\gamma}_t|\,\d t$ is Borel, as claimed. \par
Fix $(f,\pi)\in D^1(\X)$,
let $D(\pi)$ be as in \cref{lemma:measurability_enhanced_frag},
and define $F\colon [0,1]\times\Gamma(\X)\to\R$ as
$F(t,\gamma):=(f\circ\gamma)_t(\pi\circ\gamma)_t'$ for $(t,\gamma)\in D(\pi)$
and $F(t,\gamma):=0$ otherwise.
As before, we see that $F$ is Borel measurable and therefore so is
$\gamma\in\Gamma(\X)\mapsto\int_0^1 F(t,\gamma)\,\d t=\curr{\gamma}(f,\pi)$.
The thesis then follows from \cref{lemma:meas_mass}.
\end{proof}
\begin{theorem}\label{thm:representation_currents}
Let $(\X,\sfd)$ be a complete and separable metric space and $T \in \cM_1(\X)$. Then there exists a finite measure $\eta \in \mathcal{M}(\Gamma(\X))$ such that
\begin{equation}
\label{eq:main_result_representation}
    T(f,\pi)=\int \curr{\gamma}(f,\pi) \,\d \eta(\gamma)\quad\text{ and }\quad\Mass(T)=\int \ell(\gamma)\,\d \eta(\gamma)=\int\mathcal{L}^1(\mathrm{dom}(\gamma))\,\d\eta(\gamma)
\end{equation}
for $(f,\pi)\in D^1(\X)$.
In particular, for $\eta$-a.e.\ $\gamma\in\Gamma(\X)$ it holds $|\dot{\gamma}_t|=1$
at $\mathcal{L}^1$-a.e.\ $t\in{\rm dom}(\gamma)$.
\end{theorem}
\begin{remark}\label{rmk:mass_vs_mass_on_curves}
Equations \eqref{eq:main_result_representation}
implies that, for $B\subset\X$ Borel, it holds
\begin{equation}\label{eq:mass_vs_mass_on_curves}
\|T\|(B)=\int \int_{\gamma^{-1}(B)}|\dot{\gamma}_t|\,\d t\,\d\eta(\gamma)=\int\gamma_*\mathcal{L}^1(B)\,\d\eta(\gamma).
\end{equation}
Indeed, from the first equality of \eqref{eq:main_result_representation}
(and the `in particular' part of the statement), $\|T\|(B)$
is always bounded above by the right-hand side of \eqref{eq:mass_vs_mass_on_curves}
and, if the inequality was strict, summing $\|T\|(B)$ and $\|T\|(\X\setminus B)$
we would contradict the second equality in \eqref{eq:main_result_representation}.
\end{remark}

\begin{proof}
Let $\iota\colon\X\to\mathbb{B}$ be an isometric embedding into a separable Banach space $\mathbb{B}$ and consider $T':=\iota_*T\in\cM_1(\mathbb{B})$.
We may suppose ${\rm supp}(T')$ to be contained in a hyperplane in $\B$.
Fix $\varepsilon>0$.
By Theorem \ref{theorem:main_alberti_marchese}, there is a boundaryless normal current $N\in\cN_1(\mathbb{B})$
with $N\restr{\mathrm{supp}(T')}=T'$ and $\Mass(N)\leq2\Mass(T)+\varepsilon$.
Let $\eta'\in\mathcal{M}(\Lip_1([0,1],\B))$
denote the finite measure given by \cref{thm:Paolini_Stepanov} applied to $N$.
Note that it is concentrated on Lipschitz curves satisfying $|\dot{\gamma}_t|=1$ a.e.
By inner regularity of $\|T'\|$, there is a $\sigma$-compact set $E\subset{\rm supp}(T')$ on which $\|T'\|$ is concentrated.
Hence, setting $F:=\overline{E}\subset{\rm supp}(T)$, we have
\begin{equation}
\label{eq:application_albmarchese}
    T' = N\restr{F}
\end{equation}
and therefore
\begin{equation}
\label{eq:not_charging_FminusE}
    0=\|T'\|(F\setminus E)=\|N\|(F\setminus E)=\int \gamma_*\mathcal{L}^1(F\setminus E)\,\d\eta'(\gamma),
\end{equation}
and in particular $\gamma_*\mathcal{L}^1(F\setminus E)=0$ for $\eta'$-a.e.\ $\gamma$.
Recall that $\Lip_1([0,1],\X)$ is topologised by
the inclusion $\Lip_1([0,1],\X)\subset\Gamma(\X)$; see Remark \ref{rem:topologies_agree}.
Hence, denoting with $\Phi\colon\Gamma(\mathbb{B},E)\to\Gamma(\F)$ the Borel map given by
\cref{lemma:meas_restriction_frag}, we have for $(f,\pi)\in D^1(\B)$
\begin{align*}
    T'(f,\pi)&\stackrel{\eqref{eq:application_albmarchese}}{=}N(\chi_Ff,\pi)=\int\curr{\gamma}(\chi_Ff,\pi)\,\d\eta'(\gamma)
    \stackrel{\eqref{eq:not_charging_FminusE}}{=}\int_{\Gamma(\mathbb{B},E)}\curr{\Phi(\gamma)}(f,\pi)\,\d\eta'(\gamma), \\
    \Mass(T')&\stackrel{\eqref{eq:application_albmarchese}}{=}\|N\|(F)=\int \gamma_*\mathcal{L}^1(F)\,\d t\d\eta'(\gamma)
    \stackrel{\eqref{eq:not_charging_FminusE}}{=}\int_{\Gamma(\mathbb{B},E)} \Phi(\gamma)_*\mathcal{L}^1(F)\,\d\eta'(\gamma).
\end{align*}
Now, let $\tilde{\eta}\in\mathcal{M}(\Gamma(F))$ denote the pushforward under $\Phi$
of the restriction of $\eta'$ to the Borel set $\Gamma(\mathbb{B},E)$. Therefore,
\begin{equation}
    \label{eq:partial_result_representation}
    T'(f,\pi)=\int_{\Gamma(F)}\curr{\gamma}(f,\pi)\,\d \tilde{\eta}(\gamma)\quad\text{and}\quad\Mass(T')=\int_{\Gamma(F)} \gamma_*\mathcal{L}^1(F)\,\d \tilde{\eta}(\gamma)
\end{equation}
for every $(f,\pi)\in D^1(\X)$.
The measure $\tilde{\eta}$ is not yet the measure we need, since it provides a representation of $T'$ instead of $T$.\par
Let $\mathbf{p}\colon \Gamma(\iota(\X)) \to \Gamma(\X)$ be the isometry defined as $\mathbf{p}(\iota\circ \gamma):=\gamma$, which is well-defined because $\iota\colon \X\to\iota(\X)$ is a surjective isometry. Similarly, we define the isometry $p\colon \iota(\X)\to \X$ as $p(\iota(x)):=x$.
Since $\tilde{\eta}\in\mathcal{M}(\Gamma(F))\subset\mathcal{M}(\Gamma(\iota(\X)))$
and $\mathbf{p}$ is Borel, we can set $\eta:=\mathbf{p}_*\tilde{\eta}\in \mathcal{M}(\Gamma(\X))$. Moreover, it follows by the very definitions of $p$ and $\mathbf{p}$ that
\begin{equation}
\label{eq:swap_p_with_boldp}
    (\mathbf{p}(\gamma))_t=p(\gamma_t)\quad \text{for every }t \in {\rm dom}(\gamma).
\end{equation}

By \cref{lemma:restrictions_of_currents}, we may regard $T'$ as a current on the metric space $\iota(\X)\subset \B$
and therefore $T=p_*T'$.
Fix $(f,\pi)\in D^1(\X)$.
Given $\gamma \in \Gamma(\X)$, we define $G\colon\Gamma(\X)\to\R$ as $G(\gamma):=\int_{{\rm dom}(\gamma)} (f \circ \gamma)_t (\pi \circ \gamma)'_t\,\d t$,
which is Borel measurable by \cref{lemma:meas_frag_curr}.
We compute
\begin{equation*}
\begin{aligned}
    T(f,\pi)&=p_* T'(f,\pi)=T'(f \circ p, \pi \circ p)
    \stackrel{\eqref{eq:partial_result_representation}}{=}
    \int \int_{{\rm dom}(\gamma)} (f \circ p \circ \gamma)(t) (\pi \circ p \circ \gamma)'(t)\,\d t\,\d \tilde{\eta}(\gamma)\\
    &\stackrel{\eqref{eq:swap_p_with_boldp}}{=}\int \int_{{\rm dom}(\gamma)} (f \circ \mathbf{p}(\gamma))_t (\pi \circ \mathbf{p}(\gamma))'_t\,\d t\,\d \tilde{\eta}(\gamma)=\int G\circ \mathbf{p}\,\d \tilde{\eta}=\int G\,\d \eta,\\
\end{aligned}
\end{equation*}
which proves the first identity in \eqref{eq:main_result_representation}.
The second one follows from
\begin{equation*}
    \Mass(T)=\Mass(T')
    \stackrel{\eqref{eq:application_albmarchese}}{=}
    \int\mathcal{L}^1({\rm dom}(\gamma))\,\d\tilde{\eta}(\gamma)
    =
    \int\mathcal{L}^1\big({\rm dom}(\mathbf{p}(\gamma))\big)\,\d\tilde{\eta}(\gamma)
    =\int \mathcal{L}^1({\rm dom}(\gamma))\,\d\eta(\gamma),
\end{equation*}
where the last equality from the definition of $\eta$.
\end{proof}

Given a Radon measure $\mu$ on a complete and separable metric space $(\X,\sfd)$,
we let $\mathfrak{X}(\mu)$ denote the space ($L^\infty(\mu)$-normed $L^\infty(\mu)$-module)
of Weaver derivations;
see \cite{WeaverBook, Schioppa_derivations, Schioppa_currents}.
It is well-known that there is a very close connection between Weaver derivations
and $1$-currents; see \cite{Schioppa_currents}, where this is developed also for $k$-currents with $k\geq 2$.
Moreover, in \cite{Schioppa_derivations, Schioppa_currents}, it is shown that
Weaver derivations admit a representation in terms of derivatives along curve fragments,
albeit under some finite dimensionality assumptions.
It follows from \cref{thm:representation_currents} that such assumptions
can be lifted, see \cref{rmk:AR_surj_in_Der}. \par
We combine \cref{thm:representation_currents} with a disintegration theorem
\cite[Theorem 5.3.1]{AGS_gradient_flows} to obtain a pointwise representation
of Weaver derivations as a superposition of partial derivatives along curve fragments.

\begin{corollary}\label{corol:rep_derivations}
Let $(\X,\sfd,\mu)$ be a complete and separable metric measure space 
with $\mu$ nonnegative $\sigma$-finite Borel measure
and let $D \in \mathfrak{X}(\mu)$. Then there exists
a collection of Borel measures $\{\eta^x\}_{x\in\X}$ on $\overline{\Gamma}(\X)$
with the following properties.
The measure $\eta^x$ concentrated on $e^{-1}\{x\}$ for $\mu$-a.e.\ $x\in\X$,
$x\mapsto\eta^x(B)$ is Borel measurable for each Borel $B\subset\overline{\Gamma}(\X)$,
and
for any Lipschitz $f\colon\X\to\R$
it holds
\begin{align*}
D f(x)
&=\int (f\circ \gamma)'_t\,\d \eta^x(t,\gamma),
&
|D|(x)
&=\int|\dot{\gamma}_t|\,\d\eta^x(t,\gamma)=\eta^x(\overline{\Gamma}(\X)),
\end{align*}
for $\mu$-a.e.\ $x\in\X$.
\end{corollary}
\begin{proof}
Since $\mu$ is $\sigma$-finite, there is a Borel function $w\colon\X\to(0,\infty)$
such that $\mu_0:=w\mu$ is a probability measure.
Observe that $D\in\mathfrak{X}(\mu_0)$ and 
let $T$ denote the metric $1$-current induced by $D$ and $\mu_0$, i.e.\
$T(f,\pi):=\int fD\pi\,\d\mu_0$ for $(f,\pi)\in D^1(\X)$.
See \cite[Theorem 3.7]{Schioppa_currents} for a proof that it is a current and
$\|T\|=|D|\mu_0$.
Let $\eta\in\mathcal{M}(\Gamma(\X))$ be the measure given by \cref{thm:representation_currents}
and define the following finite Borel measure on $\overline{\Gamma}(\X)$
\begin{equation*}
    \overline{\eta}(B):=\int\int_{{\rm dom}(\gamma)}\chi_B(t,\gamma)\,\d t\,\d\eta(\gamma)
\end{equation*}
for $B\subset\overline{\Gamma}(\X)$ Borel.
Since $\eta$ is concentrated on curve fragments with metric derivative $1$ a.e.
(and using \cref{lemma:measurability_enhanced_frag}), we see that
$\{(t,\gamma)\in\mathit{MD}\colon |\dot{\gamma}_t|=1\}$
is a Borel set of full $\overline{\eta}$-measure; therefore
\begin{equation}\label{eq:rep_derivations_1}
    \overline{\eta}(B)=\int\int\chi_B(t,\gamma)|\dot{\gamma}_t|\,\d t\,\d\eta(\gamma)
\end{equation}
for $B\subset\overline{\Gamma}(\X)$ Borel
and moreover $e_*\overline{\eta}=\|T\|=|D|\mu_0$.
We apply \cite[Theorem 5.3.1]{AGS_gradient_flows} to obtain
a measurable family of probability measures $\{\tilde{\eta}^x\}_{x\in \X}$ on $\overline{\Gamma}(\X)$,
which disintegrate $\overline\eta$ w.r.t.\ the map $e$.
More explicitly, $\tilde{\eta}^x(\overline{\Gamma}(\X)\setminus e^{-1}\{x\})=0$ for $\mu_0$-a.e.\ $x$, and, for $B\subset\overline{\Gamma}(\X)$ Borel, $x\mapsto\tilde{\eta}^x(B)$ is Borel
and
\begin{equation*}
    \overline{\eta}(B)=\int \tilde{\eta}^x(B)|D|(x)\,\d\mu_0(x).
\end{equation*}
We then set $\eta^x:=|D|(x)\tilde{\eta}^x$,
for a fixed Borel representative of $|D|\in L^\infty(\mu_0)$.
Let $f\colon\X\to\R$ be Lipschitz and observe
that the Borel set $D(f)$, defined as in \cref{lemma:measurability_enhanced_frag},
has full $\overline{\eta}$-measure.
Let $n\in\N$ and $B\subset\{w>1/n\}$ be a Borel set.
Then $\tfrac{1}{w}\chi_B\in \mathcal{B}^\infty(\X)$ and we have
\begin{align*}
    \int_{B}Df\,\d\mu&=\int\int_{\mathrm{dom}(\gamma)}(\tfrac{1}{w}\chi_B)\circ\gamma_t(f\circ\gamma)_t'\,\d t\eta(\gamma) =\int \int_{e^{-1}\{x\}}(\tfrac{1}{w}\chi_B)(x)(f\circ\gamma)_t'\,\d\eta^x(t,\gamma)\d\mu_0(x)\\
    &=\int_B\int(f\circ\gamma)_t'\,\d\eta^x(t,\gamma)\d\mu(x)
\end{align*}
Since $B$ was arbitrary, this gives the first equality of the thesis
at $\mu$-a.e.\ $x\in \{w>1/n\}$.
Taking the union over $n\in\N$ and recalling that $w>0$ everywhere
we see that such equality holds at $\mu$-a.e.\ $x\in\X$.
The second one is obtained similarly. Indeed, if $w$ is bounded away from $0$ on the Borel set $B$
\begin{equation*}
    \int_B|D|\d\mu\stackrel{\eqref{eq:mass_vs_mass_on_curves}}{=}\int\int(\tfrac{1}{w}\chi_B)\circ\gamma_t\,\d t\d\eta(\gamma)=\int\int_{e^{-1}\{x\}}\chi_B(x)\d\eta^x(t,\gamma)\d\mu(x)
\end{equation*}
and, since $|\dot{\gamma}_t|=1$ for $\overline{\eta}$-a.e.\ $(t,\gamma)\in\overline{\Gamma}(\X)$, we also have
\begin{equation*}
\int_B|D|\d\mu=\int\int(\tfrac{1}{w}\chi_B)\circ\gamma_t|\dot{\gamma}_t|\,\d t\d\eta(\gamma)
=\int\int_{e^{-1}\{x\}}\chi_B(x)|\dot{\gamma}_t|\,\d\eta^x(t,\gamma)\d\mu(x).
\end{equation*}
By arbitrariness of $B$, we conclude.
\end{proof}
\begin{remark}
In the proof of \cref{corol:rep_derivations}, we
construct an Alberti representation $(\eta,\{\tfrac{1}{w}\gamma_*|\dot{\gamma}_t|\mathcal{L}^1\})$ of $|D|\mu$ inducing the given derivation $D\in\mathfrak{X}(\mu)$;
see \cite[Definition 2.2]{Bate2015_Str_of_meas}
for the definition of Alberti representation\footnote{Note that, contrary to \cite{Bate2015_Str_of_meas}, we do not require $\eta$ to be concentrated on biLipschitz curve fragments.}
and \cref{rmk:AR_surj_in_Der} for further comments on derivations.
Allowing $\eta$ to be $\sigma$-finite, we can take $\gamma_*(|\dot{\gamma}_t|\mathcal{L}^1)$ in place of $\tfrac{1}{w}\gamma_*|\dot{\gamma}_t|\mathcal{L}^1$; this is proven applying
\cref{thm:representation_currents} to a countable decomposition of $\mu$.
\end{remark}
\begin{remark}\label{rmk:AR_surj_in_Der}
We stress that, in \cref{corol:rep_derivations}, the way we produce a derivation from an Alberti representation
differs slightly from \cite{Schioppa_derivations}.
However, it seems likely that, using \cref{corol:rep_derivations},
one should be able to prove surjectivity of Schioppa's $\mathrm{Der}$ operator
(see \cite[Theorem 3.11]{Schioppa_derivations})
under the same assumptions as \cref{corol:rep_derivations}.
\end{remark}
\appendix
\section{Integrals of currents are currents}
\label{appendix:A}
\begin{lemma}\label{lemma:weak_star_sep_of_Lip}
Let $X$ be a separable metric space.
Then there is a countable set $D\subset \mathrm{Lip}(X)$ with the following property.
For any Lipschitz $f\colon X\to\R$ there is a sequence $f_n\in D$ converging pointwise to $f$ and satisfying $\mathrm{Lip}(f_n)\leq\Lip(f)$ for every $n$.
\end{lemma}
\begin{proof}
Let $S\subset X$ be countable and dense. Let $F\subset S$ be finite and let $D_F$ be the countable set of functions $F\to \Q$.
For each $f\in D_F$, fix an extension $\tilde{f}\colon X\to\R$ with $\Lip(\tilde{f})=\Lip(f)$.
For instance, one can set
\begin{equation*}
    \tilde{f}(x):=\min_{y\in F}f(y)+\Lip(f)d(x,y),\qquad x\in X.
\end{equation*}
Let $\tilde{D}_F:=\{\tilde{f}\colon f\in D_F\}$, $D:=\cup\{\tilde{D}_F\colon F\subset S\text{ finite}\}$ and note that they are both countable. We now show that $D$ is dense.
Let $f\colon X\to\R$ be Lipschitz and fix an enumeration $\{x_n\}_n$ of $S$.
It is easy to see from the definitions that for any $n\in\N$, there is $f_n\in D$ with $\Lip(f_n)\leq\Lip(f)$, satisfying
\begin{equation*}
\max_{1\leq i\leq n}|f_n(x_i)-f(x_i)|\leq 1/n.
\end{equation*}
It is then clear that $f_n$ converges to $f$ pointwisely on $S$.
But then, by equicontinuity, $f_n$ converges to $f$ on $X$.
\end{proof}
\begin{lemma}\label{lemma:mass_countable_sup}
Let $X$ be a complete and separable metric space.
Then there are countable sets $D\subset \Lip(X)$, $\mathcal{A}\subset\mathcal{B}(X)$, such
that for any $k\in\N_0$, $T\in\cM_k(X)$, and Borel set $B\subset X$ it holds
\begin{equation}\label{eq:mass_countable_sup}
\|T\|(B)=\sup
\sum_{i=1}^nT(\chi_{A_i\cap B},\pi_1^i,\dots,\pi_k^i),
\end{equation}
where the supremum is taken over all disjoint $\{A_i\}\subset \mathcal{A}$ and $\{\pi_j^i\}\subset D$ with $\Lip(\pi_j^i)\leq 1$.
\end{lemma}
\begin{proof}
Let $\mathcal{A}$ be an algebra of sets generated by a countable topological basis of $X$; note that it is countable.
Let $D\subset \Lip(X)$ be given by \cref{lemma:weak_star_sep_of_Lip}.
We now show that the thesis holds with these choices of $\mathcal{A}$ and $D$.
Fix a Borel set $B\subset X$ and denote with $s$ the right-hand side of \cref{eq:mass_countable_sup}.
It is clear that $s\leq \|T\|(B)$. Fix $\varepsilon>0$.
By \cite[Proposition 2.7]{AK00}, there are countably many disjoint Borel sets $B_i$ with $\cup_iB_i=B$
and $1$-Lipschitz functions $\pi_j^i$ such that
\begin{equation}
\label{ak:realisation_of_mass}
\|T\|(B)-\varepsilon<\sum_iT(\chi_{B_i},\pi_1^i,\dots,\pi_k^i).
\end{equation}
For $n\in\N$ sufficiently large, we can replace the above sum over $i\in\N$ with a sum over $1\leq i\leq n$. Also, from the continuity axiom of metric currents and \cref{lemma:weak_star_sep_of_Lip},
we can assume $\pi_j^i\in D$.
It remains to show that we can approximate the sets $\{B_i\}_{i=1}^n$ with elements of $\mathcal{A}$.
Since $\mathcal{A}$ is an algebra generating $\mathcal{B}(X)$, we have
\begin{equation*}
    \|T\|(S)=\inf\left\{\sum_i\|T\|(A_i)\colon A_i\in\mathcal{A}, S\subset\bigcup_iA_i\right\},
\end{equation*}
for any Borel $S\subset X$. Applying the above to $B_i$, we find $A_i^\delta\in\mathcal{A}$ with $\|T\|(A_i^\delta \Delta B_i) <\delta$ for $1\leq i\leq n$.
Set $E_i^\delta:=A_i^\delta\setminus \cup_{j<i}A_j^\delta$ and observe that $E_i^\delta\in\mathcal{A}$, because $\mathcal{A}$ is an algebra.
Since $B_i$ are pairwise disjoint and contained in $B$, we have $\|T\|(E_i^\delta\Delta B_i)\rightarrow 0$ and $\|T\|((E_i^\delta\cap B)\Delta B_i)\rightarrow 0$ as $\delta\rightarrow 0$.
Together with \eqref{ak:realisation_of_mass}, this implies that
\begin{equation*}
    \|T\|(B)-\varepsilon<\sum_{i=1}^nT(\chi_{E_i^\delta\cap B},\pi_1^i,\dots,\pi_k^i)\leq s,
\end{equation*}
for $\delta>0$ sufficiently small.
Since $\varepsilon>0$ was arbitrary, we have the thesis.
\end{proof}
\begin{lemma}\label{lemma:meas_mass}
Let $X$ be a complete and separable metric space, $(\Omega,\mathcal{F})$ a measurable space, and $k\in\N_0$.
Let $\Omega\ni \omega\mapsto T_\omega\in \cM_k(X)$ be a map and suppose
$\Omega \ni \omega\mapsto T_\omega(f, \pi) \in \mathbb{R}$ is measurable for each $(f, \pi)\in D^k(X)$.
Then, for each $f\in\mathcal{B}^\infty(\X)$, $\pi\in\LIP(\X)^k$, and $B\subset X$ Borel, the maps $\Omega\ni\omega\mapsto T_\omega(f,\pi)$ and $\Omega \ni \omega\mapsto \|T_\omega\|(B)\in \mathbb{R}$ are measurable.
\end{lemma}
\begin{proof}
The vector space of functions $f\colon X\to\R$ for which $\omega\mapsto T_\omega(f, \pi)$ is measurable for each $\pi\in\Lip(X)^k$ is closed under pointwise limits of equibounded sequences, and contains $\Lip_b(X)$. It therefore includes $\mathcal{B}^\infty(X)$.
Hence, by \cref{lemma:mass_countable_sup}, $\omega\mapsto \|T_\omega\|(B)$ can be written as the pointwise supremum of a countably many measurable functions.
\end{proof}
\begin{lemma}
\label{lemma:integral_of_currents_are_currents}
Let $X$ be a complete and separable metric space, $(\Omega,\mathcal{F},\mu)$ a measure space, and $k\in\N_0$. For every map $\Omega \ni \omega\mapsto T_\omega\in \cM_k(X)$ satisfying
\begin{itemize}
\item[(i)] $\omega\mapsto T_\omega(f, \pi)$ is measurable for each $(f,\pi)\in D^k(X)$;
\item[(ii)] $\int\|T_\omega\|(X)\,\d\mu(\omega)<\infty$.
\end{itemize}
we have the following.
The map $T\colon D^k(X)\to\R$ defined as  
\begin{equation*}
    T(f, \pi):=\int T_\omega(f, \pi)\,\d\mu(\omega)
\end{equation*}
is a metric $k$-current, which moreover satisfies
\begin{equation}
\label{eq:mass_estimate_integral}
    \|T\|(B)\leq\int\|T_\omega\|(B)\,\d\mu(\omega),
\end{equation}
for each $B\subset X$ Borel.
\end{lemma}
\begin{proof}
By (ii), $T(f, \pi)$ is well-defined.
It is clear that $T$ satisfies the multilinearity and locality axioms.
To see that $T$ has finite mass, observe that map $\nu(B):=\int\|T_\omega\|(B)\,\d\mu(\omega)$, $B\in\mathcal{B}(\X)$, defines a finite Borel measure on $X$ and
\[|T(f, \pi)| \le \prod_{i=1}^k \Lip(\pi_i) \int |f|\,\d\nu.\]
Thus, in particular, \eqref{eq:mass_estimate_integral} holds.

It remains to show that $T$ satisfies the joint continuity axiom.
Let $f\in \Lip_b(X)$, and $\pi^i,\pi\in\Lip(X)^k$ with $\pi^i\rightarrow \pi$ pointwisely on $X$ and $\sup_{i,j}\Lip(\pi^i_j)=:L<\infty$.
Then $|T_\omega(f, \pi^i)|\leq\|f\|_\infty L^k\|T_\omega\|(X)\in L^1(\Omega,\mu)$
and $T_\omega(f, \pi^i)\rightarrow T_\omega(f, \pi)$ for each $\omega$.
By the dominated convergence theorem, we conclude that $T(f, \pi^i)\rightarrow T(f, \pi)$
and so $T$ is a metric current.
\end{proof}
\section{Restriction of currents}
\label{appendix:B}
\begin{lemma}\label{lemma:continuity_McShane}
Let $\X$ be a metric space, $S\subset \X$ a set, and $f,g\colon S\to\R$ $L$-Lipschitz functions.
Define $\tilde{f}\colon \X\to\R$
\begin{equation*}
\tilde{f}(x):=\inf_{y\in S}f(y)+L\sfd(x,y),\qquad x\in\X,
\end{equation*}
and $\tilde{g}$ similarly.
Then $\|\tilde{f}-\tilde{g}\|_\infty=\|f-g\|_\infty$.
\end{lemma}
\begin{proof}
Since $\tilde{f},\tilde{g}$ extend $f,g$ to $\X$, it is enough to show
that $\|\tilde{f}-\tilde{g}\|_\infty\leq\|f-g\|_\infty=:M$.
We can assume $M<\infty$.
From the inequalities
\begin{equation*}
    f-M\leq g\leq f+M
\end{equation*}
we immediately deduce $\tilde{f}-M\leq \tilde{g}\leq \tilde{f}+M$, concluding the proof.
\end{proof}
\begin{lemma}\label{lemma:currents_joint_continuity_outside_null_set}
Let $\X$ be a complete metric space, $k\in\N$, $T\in\cM_k(\X)$, and assume 
the mass measure $\|T\|$ of $T$ is inner regular by compact sets.
Let $f\in\mathcal{B}^\infty(\X)$, $\pi^i,\pi\in\Lip(\X)^k$ with $\sup_{i,j}\mathrm{Lip}(\pi^i_j)<\infty$ and suppose
$\pi^i\rightarrow\pi$ pointwise on a set of full $\|T\|$-measure.
Then $T(f,\pi^i)\rightarrow T(f,\pi)$.
\end{lemma}
\begin{proof}
Fix $\varepsilon>0$ and let $K\subset \X$ be a compact set with $\|T\|(\X\setminus K)\leq \varepsilon$, such that $\pi^i\rightarrow \pi$ pointwise on $K$.
Set $L:=\sup_{i,j}\mathrm{Lip}(\pi^i_j)$ and
let $\tilde{\pi}^i$, $\tilde{\pi}$ denote the $L$-Lipschitz extension of $\pi^i\restr{K}$, $\pi\restr{K}$ constructed as in \cref{lemma:continuity_McShane}.
Since $K$ is compact and $\pi^i$ equicontinuous, $\pi^i\rightarrow \pi$ uniformly on $K$.
Then, by \cref{lemma:continuity_McShane}, $\tilde{\pi}^i\rightarrow \tilde{\pi}$ uniformly on $\X$
and so $T(f,\tilde{\pi}^i)\rightarrow T(f,\tilde{\pi})$. \par
By multilinearity, locality, and the definition of $K$, we see that
\begin{equation*}
    |T(f,\tilde{\pi}^i)-T(f,\pi^i)|\leq 2Lk\Mass(T) \|f\|_\infty\varepsilon
\end{equation*}
and similarly for $|T(f,\tilde{\pi})-T(f,\pi)|$.
Since $T(f,\tilde{\pi}^i)\rightarrow T(f,\tilde{\pi})$, we have
\begin{equation*}
    \limsup_{i\rightarrow \infty}|T(f,\pi^i)-T(f,\pi)|\leq 4Lk\Mass(T) \|f\|_\infty\varepsilon.
\end{equation*}
Letting $\varepsilon\rightarrow 0$ concludes the proof.
\end{proof}
Recall that if $\mu$ is a ($\sigma$-)finite Radon measure on a metric space $\X$,
then $\mu(\X\setminus\mathrm{spt}(\mu))=0.$

\begin{lemma}\label{lemma:restrictions_of_currents}
Let $(\X,\sfd)$ be a complete metric space and $C\subset\X$ a closed subset.
Let $T\in\cM_k(\X)$, $k\in\N$, and suppose ${\rm supp}(T)\subset C$ and
that $\|T\|$ is inner regular.
For $(f,\pi)\in D^k(C)$, set
\begin{equation*}
    \widehat{T}(f,\pi):=T(\hat{f},\hat{\pi}),
\end{equation*}
where $\hat{f},\hat{\pi}_1,\dots,\hat{\pi}_k$ are arbitrary Lipschitz extensions of $f,\pi_1,\dots,\pi_k$, respectively, with $\hat{f}$ bounded.
Then $\widehat{T}$ does not depend on the choice of extensions, $\widehat{T}\in\cM_k(C)$,
$\|\widehat{T}\|$ is inner regular, and $\iota_*\widehat{T}=T$, where $\iota\colon C\to\X$ is the inclusion map.
\end{lemma}
\begin{proof}
The functional $\widehat{T}$ does not depend on the choice of extensions $\hat{f},\hat{\pi}_j$
because of the locality property of $T$ and $\|T\|(\X\setminus C)=0$.
It is not difficult to see that $\widehat{T}$ is multilinear, local, and that
$\Mass(\widehat{T})\leq\Mass(T)$.
By \cref{lemma:currents_joint_continuity_outside_null_set}, it follows that $\widehat{T}$ satisfies
also the joint continuity axiom, and therefore is a metric $k$-current. For $(f,\pi) \in D^k(\X)$, we have
\begin{equation*}
    \iota_*\widehat{T}(f,\pi)=\widehat{T}(f \circ \iota, \pi \circ \iota)= T(\widehat{f \circ \iota}, \widehat{\pi \circ \iota})= T(f,\pi),
\end{equation*}
thus $\iota_*\widehat{T}=T$.
Also, since
$\iota_*\|\widehat{T}\|\leq\|T\|$,
$\|\widehat{T}\|$ is also concentrated on a $\sigma$-compact set
and is therefore inner regular.
\end{proof}

\begin{proposition}
\label{prop:isometry_of_currents_on_C}
Let $(\X,\sfd)$ be a complete metric space, $C\subset\X$ a closed set,
and let $\iota\colon C\to \X$ denote the inclusion map.
Then, for $k\in\N$, the pushforward operator
\begin{equation*}
\iota_* \colon \{T \in \cM_k(C):\, \|T\| \text{ is inner regular}\} \to \{ T \in \cM_k(X):\,{\rm supp}(T) \subset C,\,\|T\| \text{ is inner regular}  \}
\end{equation*}
is an isometric isomorphism,
where both spaces are endowed with the corresponding mass norm.
Moreover,
\begin{equation}\label{eq:isometry_of_currents_1}
\begin{aligned}
\iota_* \colon \{T \in \cN_k(C):\,& \|T\|+\|\partial T\| \text{ is inner regular}\} \to \\
&\{ T \in \cN_k(X):\,{\rm supp}(T) \subset C,\,\|T\|+\|\partial T\| \text{ is inner regular}  \}
\end{aligned}
\end{equation}
is an isometric isomorphism
when both spaces are equipped with either the mass norm
or normal mass norm.
\end{proposition}

\begin{proof}
    Let $Z_k(C):=\{T \in \cM_k(C):\, \|T\| \text{ is inner regular}\}$ and $Z_k(\X;C):= \{ T \in \cM_k(\X):\,{\rm supp}(T) \subset C,\,\|T\| \text{ is inner regular}  \}$. 
    Throughout, for $g\in\Lip(C)$, we let $\widehat{g}$ denote any Lipschitz extension of $g$
    to $\X$, satisfying $\|\widehat{g}\|_\infty=\|g\|_{\infty}$ and $\LIP(\widehat{g})= \LIP(g)$. If $g=(g_1,\dots, g_k)\in\Lip(C)^k$, we set $\widehat{g}=(\widehat{g}_1,\dots,\widehat{g}_k)$.
    
    Let $T \in Z_k(C)$.
    It is straigthforward to check that $\iota_* T \in Z_k(\X;C)$ and $\| \iota_* T\| \le \iota_* \|T\|$.
    Let $\Phi \colon Z_k(\X;C) \to Z_k(C)$ be the map defined as $\Phi(T):=\widehat{T}$, where $\widehat{T}$ is as in Lemma \ref{lemma:restrictions_of_currents}. The same lemma yields that $\Phi$ is well-defined and $\iota_*\Phi(T)=T$. Moreover, $\Phi(\iota_*T) = T$, because for $(f, \pi) \in D^k(C)$
    \begin{equation*}
        \Phi(\iota_*T)(f,\pi)=T(\widehat{f}\circ\iota,\widehat{\pi}\circ\iota)=T(f,\pi),
    \end{equation*}
    where we have used $(\widehat{f} \circ \iota, \widehat{\pi}\circ \iota)=(f,\pi)$. Thus, $\iota_* \colon Z_k(C) \to Z_k(\X;C)$ is linear and bijective.
    
    It remains to prove it is an isometry. We compute for $S \in Z_k(\X;C)$ and $(f,\pi) \in D^k(\X)$
    \begin{equation*}
        |\Phi(S)(f, \pi)| \le \prod_{i=1}^k \LIP(\widehat{\pi}_i) \int |\widehat{f}|\,\d \|S\| = \prod_{i=1}^k \LIP(\pi_i) \int_C |f|\,\d \|S\|
    \end{equation*}
    thus $\|\Phi(S)\| \le \|S\|$. This together with $\iota_*$ being $1$-Lipschitz gives $\Mass(\iota_*T)=\Mass(T)$ for $T \in Z_k(C)$.

    We now verify the last claim of the statement.
    Let $W_k(C)$ and $W_k(\X;C)$ denote, respectively, the left and right-hand side of 
    \eqref{eq:isometry_of_currents_1} and observe that
    $W_k(C)=\{T\in Z_k(C)\colon \partial T\in Z_{k-1}(C)\}$
    and $W_k(\X;C)=\{T\in Z_k(\X;C)\colon \partial T\in Z_{k-1}(\X;C)\}$.
    It is clear that $\iota_*(W_k(C))\subset W_k(X;C)$
    and that $\iota_*$ is an isometric embedding with either norm.
    (For the normal mass, we use the fact that pushforward and boundary operators commute.)
    Hence, we only need to prove that $\iota_*$ is onto.
    Suppose $T\in W_k(\X;C)$ and consider $\Phi(T)\in Z_k(C)$. Since $\mathrm{supp}(\partial T)\subset\mathrm{supp}(T)$,
    we have $\partial T\in Z_{k-1}(\X;C)$ and therefore $\Phi(\partial T)\in Z_{k-1}(C)$.
    From the definitions and \ref{lemma:restrictions_of_currents}
    we have $\partial (\Phi(T))=\Phi(\partial T)$, proving $\Phi(T)\in W_k(C)$ and therefore $\iota_*(W_k(C))=W_k(\X;C)$.
\end{proof}

\bibliographystyle{abbrv}
\bibliography{biblio.bib}

\end{document}